\documentclass[12pt]{article}
\usepackage{latexsym}
\usepackage{amssymb}
\usepackage{amsmath}
\usepackage{amscd}
\usepackage{amsthm}
\usepackage{mathrsfs}
\usepackage{mathabx}

\usepackage[affil-it]{authblk}
\usepackage[hidelinks]{hyperref}

\usepackage{epigraph}

\usepackage{etoolbox}
\makeatletter
\newlength\epitextskip
\pretocmd{\@epitext}{\em}{}{}
\apptocmd{\@epitext}{\em}{}{}
\patchcmd{\epigraph}{\@epitext{#1}\\}{\@epitext{#1}\\[\epitextskip]}{}{}
\makeatother

\usepackage{slashed}
\usepackage{enumerate}
\usepackage{tikz}
\usepackage{dsfont}
\usetikzlibrary{matrix}

\usepackage{tikz-cd}

\newcommand{\bigslant}[2]{{\left.\raisebox{.2em}{$#1$}\middle/\raisebox{-.2em}{$#2$}\right.}}

\newcommand{\beq}{\begin{equation}} \newcommand{\eeq}{\end{equation}}
\newcommand{\R}{\mathbb{R}} \newcommand{\C}{\mathbb{C}}
\newcommand{\moduli}{\mathcal{M}}

\newcommand{\T}{\textbf{T}}
\newcommand{\K}{\mathbb{K}}
\newcommand{\Z}{\mathbb{Z}}  
 
\newcommand{\p}{\mathbb{P}}

\newcommand{\embedd}{\hookrightarrow}
\newcommand{\sph}{\textbf{S}}

\newcommand{\G}{\mathcal{G}}

\newcommand{\twoparteq}[2]
{
	\left\{
	\begin{array}{ll}
		#1  \\
		#2
	\end{array}
	\right.
}

\newcommand{\circdist}{1}  
\newcommand{\circrad}{7/4} 

\pgfmathsetmacro{\intrad}{sqrt((\circrad)^2 - 3*(\circdist)^2/4) - \circdist/2}

\pgfmathsetmacro{\extrad}{sqrt((\circrad)^2 - 3*(\circdist)^2/4) + \circdist/2}

\definecolor{darkgreen}{RGB}{0,156,59}
\definecolor{darkyellow}{RGB}{255,233,0}
\definecolor{darkblue}{RGB}{0,39,118}

\colorlet{180}{darkblue}
\colorlet{60}{darkyellow}
\colorlet{300}{darkgreen}

\newtheorem{theorem}{Theorem}[section]
\newtheorem{corollary}[theorem]{Corollary}
\newtheorem{lemma}[theorem]{Lemma}
\newtheorem{proposition}[theorem]{Proposition}

\theoremstyle{definition}
\newtheorem{definition}[theorem]{Definition}

\theoremstyle{remark}
\newtheorem{remark}[theorem]{Remark}
\newtheorem{example}[theorem]{Example}

\numberwithin{equation}{section}

\newcommand{\mf}{\mathbf}

\newcommand{\Hom}{\textrm{Hom}}

\newcommand{\ol}{\overline}
\newcommand{\ul}{\underline}
\newcommand{\til}{\widetilde}

\begin{document}

\title{Harmonic Higgs Bundles and Coassociative ALE Fibrations}

\author{Rodrigo Barbosa%
\thanks{\texttt{rbarbosa@scgp.stonybrook.edu}}}
\affil{Simons Center for Geometry and Physics, Stony Brook University}
\date{}
\maketitle


\begin{abstract}
	Inspired by a string duality, we construct a deformation family for $G_2$-orbifolds given as total spaces of coassociative fibrations by ADE singularities over a closed and oriented smooth three-manifold $Q$. The deformations are parametrized by sections of a fiber bundle on $Q$ that can be interpreted as spectral/cameral covers associated to certain Riemannian analogs of Higgs bundles. The spectral cover picture is a unifying framework for several different approaches to coassociative fibrations appearing in the literature. Our construction generalizes, to the context of $G_2$-geometry, a well-known family of ADE-fibered Calabi-Yau threefolds whose deformations are parametrized by spectral covers of holomorphic Higgs bundles on its base Riemann surface.
\end{abstract}

\tableofcontents

\section{Introduction}

In \cite{szendroi}, Szendr\H{o}i defined and studied a deformation family for a Calabi-Yau threefold $X$ containing a curve $\Sigma$ of ADE singularities. Later, noticing that the deformations are parametrized by spectral curves of Higgs bundles on $\Sigma$, Diaconescu, Donagi and Pantev \cite{ddp} showed that, indeed, complex structure deformations are controlled by the Hitchin system of $\Sigma$: the integrable system associated to the Jacobian fibration of the family is isomorphic to the Hitchin system of $\Sigma$. This correspondence first arose in the context of B-model geometric transitions on $X$, where the Hitchin system describes the large $N$ limit of a system of holomorphic branes wrapping the exceptional cycles in a resolution of $X$ \cite{hofman}.

A natural question is: can one build a similar picture when $X$ is replaced by a $G_2$-orbifold $M_0$ containing a three-manifold of ADE singularities $Q$? This is not idle analogy, as from a physical perspective one expects the answer to be positive. This is because of a well-known duality between M-theory and Type IIA Superstring theory: M-theory compactified on $M_0$ is dual to a brane system on $Q$, and in the low energy limit the worldvolume theory of the system is described by a gauge theory on $Q$ given by the following equations \cite{modulifixing} \cite{pw}:

 \begin{align} \begin{array}{ c c l l l}
	F_A & = & -[\theta\wedge\theta]  \\
	D_A\theta & = & 0 \\
	D_A^{\dagger_k} \theta & = & 0
	\end{array} \label{pwequations} \end{align}                                            
To explain the meaning of these equations, fix a finite ADE group $\Gamma \leq SU(2)$, $\mathfrak{g}_c$ the complex simple Lie algebra McKay-dual to $\Gamma$, $G_c$ its complex reductive Lie group and $G \leq G_c$ the compact real form with Lie algebra $\mathfrak{g}$. Then, for a principal $G_c$-bundle $\mathcal{P}_{G_c} \to Q$, $k$ is a metric on $\mathcal{P}_{G_c}$ inducing a principal  $G$-subbundle $\mathcal{P}_G \subset \mathcal{P}_{G_c}$, $F_A$ is the curvature of a $G$-connection $A$ on $Ad(\mathcal{P}_{G_c})$ that preserves $k$, $\theta \in \Omega^1(Q, Ad(\mathcal{P}_G)^\perp )$ is a real analogue of a Higgs field, and $D_A^\dagger = \star D_A \star$ is the $k$-adjoint of $D_A$.  

``Duality'' in this context means that we expect the moduli spaces of the two theories to be isomorphic. The moduli space of M-theory parametrizes ``complexified'' $G_2$-structures on (a desingularization of) $M_0$. Hence we expect that deformations of the $G_2$-structure can be somehow captured in the moduli space of solutions to \ref{pwequations}.
	
	Equations \ref{pwequations} have been known since the seminal works of Donaldson \cite{pato} and Corlette \cite{corlette}. They were re-derived in \cite{modulifixing} and \cite{pw} in the context of the aforementioned duality via a dimensional reduction of Hermitian-Yang-Mills instantons on $T^*Q$, which is the gauge-theoretic description of the dual Type IIA theory. See \cite{acharya} for foundational work on this, and \cite{sakura}, \cite{galerona} for recent developments. Following terminology in physics, we will refer to the first two equations in \ref{pwequations} as the \emph{F-terms}, and the third equation as the \emph{D-term}.
	
The F-terms are equivalent to the condition that the $G_c$-connection $\mathbb{A} := A+i\theta$ is flat. Donaldson and Corlette showed that whenever the monodromy representation of $\mathbb{A}$ has reductive Zariski closure, one can solve the D-term within the $G_c$-gauge orbit of $\mathbb{A}$. Moreover, if $\mathbb{A}$ is irreducible, the solution is unique. This establishes a bijection: 

\begin{align} \bigslant{ \left\{ \begin{array}{ c c l l l}
	F_A + [\theta\wedge\theta] & = 0 \\
	D_A\theta & = 0 \\
	D_A^{\dagger_k} \theta & = 0
	\end{array} \right\} }{\mathcal{G}^k} & \cong \bigslant{ \left\{ F_{\mathbb{A}} = 0 \right\}^{\text{irred}} }{\mathcal{G}^\C} \label{coin} \end{align}
	where $\mathcal{G}^\C$ is the group of automorphisms of $E$ seen as a $G_c$-vector bundle, and $\mathcal{G}^k$ is the subgroup that preserves the metric $k$.

This result can be understood as an infinite-dimensional version of the Kempf-Ness theorem equating symplectic quotients with GIT quotients \cite{kempfness}. Here we interpret the D-term as a moment map condition for the action of $\mathcal{G}^k$ on the space of all complex connections on $E$. Denoting by $\widetilde{Q}$ the universal cover of $Q$, a hermitian metric $k$ on $E$ can be seen as a section of an associated $\pi_1(Q)$-bundle $\widetilde{Q}\times_{\rho_\mathbb{A}}G_c/G$, where $\rho_\mathbb{A}$ is the monodromy representation of $\mathbb{A}$. Then $k$ solves \ref{coin} if and only if the associated section extremizes the $L^2$ energy, and for this reason such a metric is called \emph{harmonic}.  Accordingly, we will refer to a solution of \ref{pwequations} as a \emph{harmonic Higgs bundle}.


In this paper, we will construct a deformation family of \emph{closed} $G_2$-structures for a $G_2$-orbifold fibered by coassociative ADE singularities:

\beq \C^2/\Gamma \embedd M_0 \to Q \label{tulia} \eeq
i.e., $M_0$ has a closed $G_2$-structure $\varphi$ supported outside the zero section such that $\varphi|_{\C^2/\Gamma} = 0$. We will show that the deformations preserving the structure of a coassociative fibration are parametrized by geometric objects called \emph{spectral covers} (or, more generally, \emph{cameral covers}) associated to \emph{commuting} solutions of \ref{pwequations}, i.e. those satisfying 

\beq [\theta\wedge\theta]=0 \label{comutano} \eeq 

Let us explain in more detail. Abstractly, a \emph{cameral structure}\footnote{In \cite{gaitsgory}, this is called a \emph{Higgs structure}. Unfortunately, that terminology is not convenient for our purposes, as we consider solutions of \ref{pwequations} to be the basic objects. Since \cite{gaitsgory} only concerns \emph{holomorphic} Higgs bundles, our choice will lead to no confusion.} for $\mathcal{P}_{G}$ is a smooth section $s$ of regular $G$-centralizers of the Grassmannian bundle $Gr(r,Ad(\mathcal{P}_G)) \to Q$, where $r$ is the rank of $\mathfrak{g}$. In other words, its image consists of a family of abelian subalgebras  of the form $\mathfrak{c}_\mathfrak{g}(x)$, where $x \in \mathfrak{g}$ is regular. Since $G$ is simple, any regular element is semi-simple (i.e., diagonalizable) and hence $s$ defines a sub-bundle $\mathfrak{H} \subset Ad(\mathcal{P}_G)$ of Cartan subalgebras $\mathfrak{h} \subset \mathfrak{g}$. Crucially, even though every fiber of $Ad(\mathcal{P}_G)$ can be identified with $\mathfrak{g}$, there is no canonical identification between the fibers of $\mathfrak{H}$ to a fixed Cartan $\mathfrak{h}$. Now, since $G$ is compact and simple, the local components $(\theta_1,\theta_2,\theta_3)$ of any $\theta \in \Omega^1(Q,Ad(\mathcal{P}_G))$ are semi-simple. Furthermore, if $\theta$ satisfies \ref{comutano}, then at every point $q \in Q$ there is $g_q \in G$ that conjugates $(\theta_1(q),\theta_2(q),\theta_3(q))$ to $\mathfrak{H}_q$. Since an adjoint orbit intersects a Cartan in a Weyl orbit, the element $g_q$ is only well-defined up to the action of $W$ on $\mathfrak{H}_q$. Thus, we can see our $\theta$ as an element of $\Omega^1(Q,\mathfrak{H}/W)$, or equivalently, a section $\theta : Q \to T^*Q\otimes\mathfrak{H}/W$. 

Denoting by $\mathcal{E} := T^*Q\otimes \mathfrak{H}$, there is a covering map $e: \mathcal{E} \to \mathcal{E}/W$. Consider the pullback $c_\theta = \theta^*(e)$. It defines a Galois $|W|$-to-$1$ cover $c_\theta : \widetilde{Q}_\theta \to Q$ called the \emph{cameral cover} of $\theta$. By choosing a linear representation of $G$, one can construct the cameral cover geometrically as the space of ordered eigenvalues of $\theta$. There is also an induced \emph{spectral cover} parametrizing (unordered) eigenvalues of $\theta$ in that representation. Geometrically, a spectral cover is a Lagrangian subspace of $p: T^*Q \to Q$ such that the restriction $p|_{s(Q)} : s(Q) \to Q$ is a branched finite cover of degree $|\Gamma|$.

Roughly speaking, $\theta$ defines a $G_2$-deformation $M_\theta$ of $M_0$ as follows: a Higgs field $\theta \in \Omega^1(Q,\mathfrak{H}/W)$ can be interpreted as a $G$-invariant, $SO(3)$-equivariant map:

\[ \theta : Fr_Q \times \mathcal{P}_G \to Z \]
into the parameter space $Z \cong \R^3\otimes \mathfrak{h}/W$ of smoothings of $\C^2/\Gamma$, where the volumes of exceptional spheres are dictated by the eigenvalues of $\theta$. Therefore, $\theta$ picks up a profile of smoothings of $\C^2/\Gamma$ over $Q$, and the construction depends only on the cameral cover of $\theta$. We then show that the remaining equations in \ref{pwequations} imply that $M_\theta$ carries a coassociative fibration to $Q$. Note that the singular set of $M_\theta$ consists exactly of the locus in $Q$ where $\theta$ fails to be regular. This set has measure zero and in fact is expected to have codimension at least two in $Q$. This is consistent with physical considerations, since codimension six and seven singularities on $G_2$-manifolds generate chiral fermions in four dimensions \cite{acharyawitten}.

In order for the construction to be useful, one needs to know examples of $M_0$ as above. We will be interested in situations where $M_0 := \mf{V}/\Gamma$, with $\mf{V}$ the total space of a rank $1$ quaternionic vector bundle over $Q$ with a $\Gamma$-action and a compatible connection. Then a fibration as in \ref{tulia} will have a $G_2$-structure $\varphi_0$ whenever it admits a structure we call a \emph{twisting section}. In a nutshell, this consists of a $\Gamma$-invariant global section $\eta_0$ of $T^*Q \otimes \Lambda^{2,+}T^*_{\text{vert}}\mf{V}$, and $\varphi_0$ will be closed or co-closed whenever $\eta_0$ is. Examples of this kind are easiest to come by when the metric connection $\delta$ is \emph{flat}, and although our main construction works in greater generality, all examples known to the author are of this form. Several such examples will be described in section \ref{adeg2platy}, including one that works for $\Gamma = \Z_n$ for any $n$.

The fundamental group of compact flat three-manifolds are extensions of $\Z^3$ by a finite group $H_\delta$, and this special nature makes the relevant moduli space - the \emph{character variety}, or more generally, the \emph{character stack} - an interesting algebraic space that often can be explicitly computed. In such cases, the resulting $G_2$-spaces have infinite fundamental group, so their holonomy cannot be upgraded from $SU(2)\rtimes H_\delta$ to $G_2$ - in fact, such spaces often appear in local models for \emph{compact} $G_2$-manifolds, as in Joyce's foundational work \cite{joyce}. Nevertheless, our main point here is that the dual description involving character varieties and spectral covers should provide a new method for constructing more interesting examples - e.g., taking $Q=\sph^3$ and $\Gamma=\Z_2$, if one can construct a spectral cover intersecting $Q$ along a knot $K \subset \sph^3$, then on the $G_2$ side one should get a local model for a coassociative Kovalev-Lefschetz fibration \cite{donaldson} \cite{kovalev}.

Before we state our main result, let us describe a somewhat simple example. Let $\G_6$ be the (unique up to affine isomorphism) compact orientable flat Riemannian $3$-manifold with holonomy $\K := \Z_2\times \Z_2 = \langle \alpha, \beta \rangle$. One can construct a flat $\K$-bundle of $A_1$-singularities $N_0 \to \G_6$ and a $G_2$-structure $\varphi$ on $N_0$ making the fibers coassociative (see example \ref{exemplodobob} below). According to Joyce \cite{joyce} there are three topologically distinct smoothings $Y_i$ of $N_0$ that retain this condition\footnote{I.e., whose fibers admits a $\K$-action asymptotic to the original action on $M_0$.}:

\begin{enumerate}
\item $Y_1$, obtained by blowing up the fibers. The exceptional curve is a $\C\p^1$ and the $\K$-action is such that $\alpha$ reverses orientation but $\beta$ does not.
\item $Y_2$, obtained by a deformation of $\C^2/\Z_2$ replacing the singularity by a totally real $\sph^2$. The $\K$-action is such that both $\alpha$ and $\beta$ reverse orientation.
\item $Y_3$, obtained by a deformation of $\C^2/\Z_2$ replacing the singularity by a totally imaginary $\sph^2$, is diffeomorphic to $Y_2$. Here only $\beta$ reverses orientation.
\end{enumerate} 

Each of the three smoothings defines a one-parameter family of $G_2$-structures parametrized by the volume of the exceptional sphere. Thus, the local moduli space of $G_2$-deformations $\moduli_{G_2}$ is given by three copies of $\R_+$ touching at a point. Away from the orbifold point, the M-theory moduli space $\moduli^\C_{G_2}$ is a torus fibration\footnote{When $M$ is a smooth compact $G_2$-manifold, Karigiannis and Leung \cite{leung} proved that $\moduli^\C_{G_2}$ is in fact K\"ahler and the fibration is Lagrangian.} $\moduli^{\C}_{G_2} \to \moduli_{G_2}$. Thus, the connected component of $\moduli^\C_{G_2}$ probed by Joyce consists of three copies of $\C$ touching at a point.

On the other hand, if the duality is to be believed, one would expect that the $SL(2,\C)$-character variety of $\G_6$ also contains a component consisting of three copies of $\C$ touching at a point. We will show in section \ref{duality} that this is in fact the case, and that moreover the other components consist of three isolated points - conjecturally corresponding to three yet unknown rigid $G_2$-orbifolds.

To state our main result, we need to introduce a class of orbifolds with $G_2$-structure whose deformations are to be described by harmonic Higgs bundles. We call the relevant objects \emph{coassociative ADE fibrations} (see definition \ref{coassade}). Essentially, it consists of an orbibundle $\overline{p} : M_0 \to Q$ obtained as a global quotient $M_0 = \text{tot}(\mf{V})/\Gamma$ of a $Spin(4)$-vector bundle $p: \mf{V} \to Q$ with some additional structure: notably, a \emph{twisting form} $\eta_0 \in \Omega^1(Q,\Lambda^{2,+}\mf{V})$ and a $\eta_0$-compatible Ehresmann connection $\mf{H}_0$ on $\overline{p}$. Additional details can be found in section $3$, where we also present a general method for building examples, and discuss some explicit ones - including $N_0$ above. The data consisting of $(p,\Gamma)$ also allows us to construct in section \ref{mainsection} a bundle of Cartan subalgebras $\mathfrak{H} \to Q$ where our harmonic Higgs bundles will live. For our setup, this will turn out to be a trivial bundle. Accordingly, we will be interested in studying \ref{pwequations} on a topologically trivial bundle $\mathcal{P}_{G_c} \to Q$.

Our main result can be stated as follows (see Theorem \ref{teoremito}):

\begin{theorem}
	Let $Q$ be a closed, oriented smooth three-manifold and $\overline{p} : M_0 \to Q$ a coassociative ADE fibration of type $\Gamma$. Let $\mathfrak{g}$ be the compact real Lie algebra McKay dual to $\Gamma$, $\mathfrak{h} \subset \mathfrak{g}$ a Cartan subalgebra, $W$ the Weyl group, and $\mathcal{P}_{G_c} \to Q$ the trivial principal $G_c$-bundle.
	
Then there is a deformation family of $7$-orbifolds with closed $G_2$-structures:
	
	\beq f: \mathcal{F} \to \mathcal{B} \eeq
	with central fiber $(M_0,\varphi_0)$. 
	
	The family has the following properties:
	
	\begin{enumerate} 
	
	\item For $s \in \mathcal{B} \setminus \left\{ 0 \right\}$, $(M_s := f^{-1}(s), \varphi_s)$ is a coassociative fibration over $Q$ whose fibers are generically hyperk\"ahler ALE spaces of type $\Gamma$.
	
	\item The base $\mathcal{B}$ parametrizes cameral covers associated to commuting solutions of the F-terms of \ref{pwequations}.  
		\end{enumerate}
	
Moreover, if $\varphi_0$ comes from a harmonic section $h_0 : Q \to H^2(M_0/Q,\R)$  and the D-term is satisfied, then $\varphi_s$ also comes from a harmonic section $h_s : Q \to H^2(M_s/Q,\R)$.
\label{teoremaum} \end{theorem}

The meaning of the last phrase in the theorem is that the D-term does not quite see the torsion-free condition $d*\varphi =0$, but rather a weaker condition which is entirely analogous to the \emph{adiabatic limit} of $d*\varphi=0$ in the context of coassociative $K3$-fibrations \cite{donaldson}. This seems to indicate that harmonic Higgs bundles describe some sort of adiabatic version of integrable $G_2$-structures, but unlike the compact case where the adiabatic limit is obtained by shrinking the fibers, in the ALE setup it should be thought as an expansion of the base $Q$.

There are two key ideas to prove the theorem: the first is to recast the data of a $G_2$-structure in a language more appropriate for deformation theory. We use Donaldson's description \cite{donaldson} of $G_2$-structures on coassociative fibrations to reformulate the problem in terms of deformations of a tuple $(\eta_0, \mu_0, \mf{H}_0)$ on $M_0 \to Q$ consisting of a connection $\mf{H}_0$, a vertical hyperk\"ahler triple $\eta_0$ and a horizontal lift $\mu_0$ of a volume form on $Q$, which are required to satisfy some compatibility conditions. The second idea is to construct a vector bundle $t: \mathcal{E} \to Q$ and a family of ALE spaces $u: \mathcal{U} \to \text{tot}(\mathcal{E})$, locally isomorphic to Kronheimer's celebrated construction \cite{kronheimer1}, as an ambient space where our deformation problem unfolds. In particular, up to the action of $W$, $\mathcal{B}$ will be identified with a subspace of sections of $\mathcal{E}$.

Commuting solutions to \ref{pwequations} are rather special, so in most situations Theorem \ref{teoremaum} will only apply to the subset of $Q$ where $[\theta\wedge\theta]=0$ holds. One could ask whether the result can be extended to $q \in Q$ where $[\theta(q)\wedge\theta(q)] \neq 0$. In this situation, there is no interpretation of $\theta(q)$ as an element of the unfolding space of $\C^2/\Gamma$, hence the fiber over $q$ cannot be smoothed to a complete hyperk\"ahler Ricci-flat ALE space. However, since such solutions are still expected to describe $G_2$-manifolds, it is natural to conjecture that the fiber can be smoothed to a non-ALE or incomplete hyperk\"ahler Ricci-flat space, and that it will be smooth if $\theta$ is regular. We will not elaborate further on this issue, leaving it to future work. As such, in this paper ALE spaces are always assumed to be hyperk\"ahler.

The paper is organized as follows: in section \ref{flat} we briefly review the classification of compact oriented flat Riemannian three-manifolds and provide an algebraic description of their character varieties with values in a reductive group. These results are only needed to compute examples, and the reader who is mainly interested in the main construction can safely skip this section. The core of the paper consists of sections \ref{adeg2platy} and \ref{deformation}. In section \ref{adeg2platy} we introduce a general construction of \emph{coassociative ADE fibrations} over a closed oriented $3$-manifold $Q$ and describe some examples. In section \ref{deformation} we provide our main construction (see Theorem \ref{teoremito}): a deformation family for coassociative ADE-fibrations. The family consists of generically coassociative ALE fibrations with respect to closed and adiabatically co-closed $G_2$-structures, and are parametrized by ``harmonic'' cameral covers associated to commuting solutions of \ref{pwequations}. In section \ref{duality} we present computations for the example $N_0$ above; in particular, we prove that the $SL(2,\C)$-character variety of $\G_6$ is the complexification of the deformation space of $G_2$-structures on $N_0$. In section \ref{spectrality} we discuss the spectral correspondence for commuting solutions of equations \ref{pwequations}. Section \ref{remakedaremarca} is a brief discussion on how the spectral/cameral cover provides a unified framework for recent constructions of coassociative fibrations, where we also conjecture that the singularity profile of the $G_2$-spaces $M_\theta := f^{-1}(\theta )$ consists of the isolated ramification points of the spectral cover associated to $\theta$.

In a forthcoming companion paper \cite{barb}, we will provide a second algebraic characterization of $\moduli^\C_{G_2}$ for our main example $N_0$, given by a Hilbert scheme of points on a singular threefold. While it is unclear if this second construction will be useful in $G_2$-geometry, it suggests a new interpretation of SYZ Mirror Symmetry in terms of moduli spaces of harmonic Higgs bundles.


A first version of our work dealt with the special case of $Q$ a \emph{flat} Riemannian $3$-manifold. After that work was completed we were informed about the work of Joyce and Karigiannis \cite{joycekarigiannis}, which was a key motivation to pursue the general setup presented here. Essentially, our construction can be seen as phrasing (conjectural) generalizations of the construction of \cite{joycekarigiannis} in the language of \cite{donaldson}. For a more detailed explanation, see section \ref{remakedaremarca}.

\section*{Acknowledgements}

This work comprised part of my Ph.D. research at the University of Pennsylvania, and as such I would like to thank my supervisor Tony Pantev for initially proposing this project, for his steady guidance through the years and for the numerous discussions that essentially shaped this work. I am also grateful to Simon Donaldson for several important discussions that greatly clarified the scope of the constructions in this paper and allowed me to generalize the original results considerably. Finally, thanks to Bobby Acharya, Mirjam Cveti\v{c}, Lorenzo Foscolo, Jonathan Heckman, Alex Kinsella, Craig Lawrie, Dave Morrison, Pavel Safronov, Sakura Sch\"afer-Nameki, Ethan Torres and Gianluca Zoccarato for discussions related to this work. 

\section{Character Varieties of Flat 3-Manifolds} \label{flat} 

\subsection{Flat Riemannian Geometry}
\label{flatland}

Let $\text{Iso}(\R^n)$ denote the group of isometries of $\R^n$ endowed with its standard Euclidean structure. Recall that a subgroup $\pi \leq \text{Iso}(\R^n)$  is called \emph{crystallographic} if it is discrete and cocompact (i.e, $Q^n := \R^n/\pi$ is compact). It is called \emph{torsion-free} if it acts freely. A torsion-free crystallographic subgroup is called \emph{Bieberbach}. Clearly $\pi$ is crystallographic if and only if $Q^n$ is a compact flat orbifold, and $\pi$ is Bieberbach if and only if $Q^n$ is a compact flat manifold. We refer to $Q^n$ as a Bieberbach space. Any crystallographic group $\pi$ fits into a short exact sequence

\beq 0 \rightarrow \Lambda \rightarrow \pi \to H \to 1 \label{exactbieber}\eeq
where $H$ is a finite group called the \emph{monodromy} of $\pi$ and $\Lambda$ is a free abelian $H$-module. Isomorphism classes of crystallographic groups are classified by the group cohomology $H^2(H,\Lambda)$.

Let $\T^n$ denote the flat $n$-torus. \emph{Bieberbach's theorem} states that:
	\begin{enumerate}
		\item There is a finite normal covering map $\T^n \to Q^n$ which is a local isometry.
		
		\item Two Bieberbach spaces of the same dimension and with isomorphic fundamental groups are affinely isomorphic.
		
		\item There are finitely many affine classes of Bieberbach spaces of dimension $n$.
	\end{enumerate}

We note that part $3$ essentially follows from the fact that the number of exact sequences \ref{exactbieber} is bounded by the order of the finite group $H^2(H,\Lambda)$.

Clearly, $\R^n$ is the universal cover of $Q^n$, and $\pi_1(Q^n) = \pi$. The first part of Bieberbach's theorem implies that the $H$-action on $\Lambda \cong \pi_1(\T^n)$ is induced from a free $H$-action on $\T^n$ such that $Q^n \cong \T^n/H$. It is clear that $\T^n$ is also a Bieberbach manifold, albeit with trivial monodromy. For this reason, we call $\T^n$ the \emph{monodromy cover} of $Q^n$. The existence of the monodromy cover strongly constrains the possible holonomies of Bieberbach spaces. This is in stark contrast with the theory for \emph{non-compact} flat Riemannian manifolds: it is a theorem of Auslander and Kuranishi that every finite group is the holonomy group of some flat manifold.

Following standard terminology \cite{conway} \cite{szczepanski}, we call a three-dimensional Bieberbach manifold $Q$ a \emph{platycosm}. There are $10$ affine equivalence classes of platycosms, $6$ of which are orientable. To distinguish them it suffices to consider their monodromy groups $H_Q$, and the classification goes as follows:
	
	\begin{itemize}
		\item $\G_1$ is the flat three-torus $\T^3$, so the monodromy is trivial: $H_{\G_1}=\left\{ 1 \right\}$
		\item $\G_2$ with $H_{\G_2} \cong \Z_2$
		\item $\G_3$ with $H_{\G_3} \cong \Z_3$
		\item $\G_4$ with $H_{\G_4} \cong \Z_4$
		\item $\G_5$ with $H_{\G_5} \cong \Z_6$
		\item $\G_6$ with $H_{\G_6} \cong \Z_2 \times \Z_2$
	\end{itemize}

The space $\mathcal{G}_6$ will be particularly important for us. It is known in the literature as the \emph{Hantzsche-Wendt manifold} or \emph{didicosm}. An explicit description for $H_{\G_6}$ is:

\beq H_{\G_6} = \left\langle A = \left[ \begin{array}{ccc} 1 & 0 & 0 \\
	0 & -1 & 0 \\
	0 & 0 & -1
\end{array} \right] ,
B= \left[  \begin{array}{ccc} -1 & 0 & 0 \\
	0 & 1 & 0 \\
	0 & 0 & -1
\end{array} \right] \right\rangle \subset SO(3) \label{hwmatrizes} \eeq

Moreover, the lattice $\Lambda_{\G_6}$ has isometry group:

\beq 
\left\langle  \left( A, \left[ \begin{array}{c} 1/2  \\ 0  \\ 0 \end{array} \right] \right) ,
\left( B, \left[ \begin{array}{c} 0  \\ 1/2  \\ 1/2 \end{array} \right] \right) 
\right\rangle \subset SO(3) \ltimes \R^3 = \text{Iso}^+(\R^3)																											 	
\eeq


\subsection{Character Varieties} \label{charvarieties}
Fix a complex reductive algebraic group $G_c$ with Lie algebra $\mathfrak{g}_c$, $\mathfrak{h}_c \subset \mathfrak{g}_c$ a Cartan subalgebra and $W$ the Weyl group. Let $G$ be the compact real form of $G_c$, with Lie algebra $\mathfrak{g}$ and Cartan subalgebra $\mathfrak{h}$.

Consider a finitely generated group $\pi = \langle g_1,\ldots,g_k \rangle$. Let $Hom(\pi,G_c)$ be the \emph{representation variety}. Since $G_c$ is a subgroup of $GL(n,\C)$, its relations together with those of $\pi$ determine $Hom(\pi,G_c)$ as an affine subvariety of $G_c^k$. If $I \subseteq \C[z_1,\ldots , z_k]$ is the ideal generated by those relations, one can consider the \emph{representation scheme}:

\[ \mathfrak{X}(\pi,G_c) = \text{Spec}(\C[z_1,\ldots,z_k]/I) \]
This scheme is independent of the presentation of $\pi$ up to canonical isomorphism. The representation variety is the reduced scheme of $\mathfrak{X}(\pi,G_c)$.

Let $Z(G_c)$ be the center of $G_c$. The adjoint group $G^\prime_c = G_c/Z(G_c)$ acts by conjugation on $Hom(\pi,G_c)$. The GIT quotient is the \emph{character variety}:

\[ \text{Char}(\pi,G_c) = Hom(\pi,G_c)//G^\prime_c \]

This is obtained by considering only representations with closed $G^\prime_c$-orbit. These are exactly the completely reducible representations \cite{sikora}. The \emph{character scheme} is:

\[ \mathfrak{X}(\pi,G_c) = \text{Spec} \left( \C[z_1,\ldots,z_k]/I \right)^{G^\prime_c} \]
and its reduced scheme is $\text{Char}(\pi,G_c)$.

\subsection{Character Varieties of Bieberbach groups} \label{charvarbieber}

By equivalence \ref{coin}, the moduli space of solutions to equations \ref{pwequations} is given by flat $G_c$-bundles on $Q$, where $G_c \leq GL(n,\C)$. In other words, the moduli space is the character variety $\text{Char}(Q,G_c)$. In order to describe this space for a general platycosm, we first consider the case $Q = \T^3$.

To understand the space $\text{Char}(\T^3,G_c)$, it is worthwile to first understand simpler examples. Let $\T^1 = \sph^1$ be the one-torus. Then $\text{Char}(\sph^1,G_c)$ is obtained by looking at the image of the generator of $\pi_1(\sph^1)$. Thus:

\beq Char(\sph^1,G_c) = G_c//G'_c \cong  \mathcal{T}_c/W \eeq
where $\mathcal{T}_c$ is the maximal torus of $G_c$. The isomorphism with $\mathcal{T}_c/W$ is essentially a consequence of Chevalley's isomorphism $\C[\mathfrak{g}_c]^{G'_c} \cong \C[\mathfrak{h}_c]^W$, see \cite{steinberg}.

Consider now a two-torus $\T^2$. We first consider the case of the compact group $G$. Then $\text{Char}(\T^2,G)$ is given by two commuting elements in $G$ up to conjugation. Let $g \in G$ and $h \in C_G(g)$, the centralizer of $g$. A theorem of Bott says that if $G$ is simply-connected, then the centralizer $C_G(g)$ is connected. Thus, we can first conjugate $g$ to a maximal torus $\mathcal{T}$ of $G$ and then conjugate $h$ to the torus of $C_G(g)$, which by connectedness is just $\mathcal{T}$. The net result is that $g$ and $h$ can be simultaneously conjugated to lie on the maximal torus $\mathcal{T}$. The maximal tori are conjugated by elements of the Weyl group $W$. Hence the character variety is:

\beq \text{Char}(\T^2,G) = \bigslant{\mathcal{T}\times \mathcal{T}}{W} \label{auaua}\eeq

For a three-torus, $\text{Char}(\T^3,G)$ is now given by three commuting elements modulo conjugation. So now we need to determine all possible configurations of $g, h, k \in G$, with $g \in \mathcal{T}$ and $h, k \in C_G(g)$, i.e., the \emph{moduli space of commuting triples}. This problem was solved by Borel, Friedman and Morgan \cite{borel} and Kac and Smilga \cite{kac}, who showed that the classification of commuting triples $(g,h,k)$ is essentially determined by the fundamental groups of the centralizers $C_G(g), C_G(h), C_G(k)$. A commuting triple $(g,h,k)$ whose semi-simple part of the centralizer is simply-connected can always be conjugated to the maximal torus, giving one of the components of the moduli space:

\beq \text{Char}^0(\T^3,G) = \bigslant{\mathcal{T}\times \mathcal{T} \times \mathcal{T} }{W} \label{auauaua} \eeq

However, there are also \emph{non-trivial} commuting triples. This happens when $G$ has elements whose semi-simple part of the centralizer has torsion. These extra commuting triples produce new connected components in the character variety. Essentially, torsion in $\pi_1(C_G(g))$ occurs when the root system of $\mathfrak{h}$ admits non-trivial \emph{coroot integers}. Each divisor of a coroot integer is called a \emph{level} $\ell$, and each $\ell$ determines a subtorus $\mathcal{T}_\ell$ of $\mathcal{T}$ given by the intersection of the kernels of the roots whose coroot integers are \emph{not} divisible by $\ell$. The torus $\mathcal{T}_\ell$ has an associated Weyl group $W_{\mathcal{T}_\ell} := N_G(\mathcal{T}_\ell)/C_G(\mathcal{T}_\ell)$. Here $N_G$ denotes the normalizer. 

Each $\ell$ determines $\phi(\ell)$ connected components for the character variety, where $\phi$ is Euler's totient function; each connected component is given by:

\beq \bigslant{\mathcal{T}_\ell\times \mathcal{T}_\ell \times \mathcal{T}_\ell}{W_{\mathcal{T}_\ell}} \eeq

In particular, for $G=SU(n)$, the only allowed level is $\ell=1$ and there are no non-trivial commuting triples. 

Now, even for $Q=\T^2$, the problem of computing $\text{Char}(Q,G_c)$ when $G_c$ is an affine reductive group over $\C$ is considerably harder, in part because Bott's theorem fails without the compactness assumption. However, it is known \cite{florentino} that if $A$ is a finitely generated abelian group, then $Hom(A,G)/G$ is a strong deformation retract of $\text{Char}(A,G_c)$. So one could still hope that \ref{auaua} and \ref{auauaua} can be generalized to classical complex groups $G$ by replacing $\mathcal{T}$ by a maximal complex torus $\mathcal{T}_c$ of $G$. 

Sikora \cite{sikorabeliano}, extending a result of Thaddeus \cite{thaddeus}, showed that this is indeed the case. To state Sikora's theorem, consider the map:

\beq \widetilde{\chi} : \mathcal{T}^n_c \cong Hom(\Z^n,\mathcal{T}_c) \embedd Hom(\Z^n,G_c) \to \text{Char}(\T^n,G_c) \eeq

Since $W$ acts by outer automorphisms on $\mathcal{T}_c$ extending to inner automorphisms of $G_c$, it follows that $\widetilde{\chi}$ is $W$-invariant and hence descends to a map:

\beq \chi : \mathcal{T}^n_c/W \to \text{Char}(\T^n,G_c) \eeq
with $W$ acting diagonally.

\begin{theorem} \emph{\cite{sikorabeliano}}  \begin{enumerate} 
\item $\text{Char}^0(\T^n,G_c) := \chi(\mathcal{T}^n_c/W)$ is an irreducible component of $\text{Char}(\T^n,G_c)$
\item For all $G_c$ and $n$, $\chi : \mathcal{T}^n_c/W \to \text{Char}^0(\T^n,G_c)$ is a normalization map, and is an isomorphism if $G_c$ is a classical group.
\item If $G_c = GL(n,\C), SL(n,\C)$ or $Sp(n,\C)$, then $\text{Char}^0(\T^n,G_c) = \text{Char}(\T^n,G_c)$
\end{enumerate}

\end{theorem}

In particular, we have homeomorphisms:

\beq \text{Char}\big( \T^3,SL(n,\C)\big) \cong \bigslant{\big( (\C^*)^{n-1}\big)^3}{\Sigma_n} \label{chartorus} \eeq
and similarly:

\beq \text{Char}\big( \T^3,SU(n)\big) \cong \bigslant{\big( U(1)^{n-1}\big)^3}{\Sigma_n} \label{chartorussun} \eeq
where $\Sigma_n$ is the symmetric group.

We now move on to a general platycosm $Q$. Let again $\pi = \pi_1(Q)$. The exact sequence 
\[ 1 \to \Lambda \to \pi \stackrel{q}{\to} H \to 1 \]
induces another exact sequence:

\beq 1 \rightarrow \Hom(H, G_c) \rightarrow \Hom(\pi,G_c) \stackrel{\ol{r}}{\to} \Hom(\Lambda,G_c) \eeq
which in turn descends to maps between character varieties:

\beq \text{Char}(H,G_c) \to \text{Char}(\pi,G_c) \stackrel{r}{\to} \text{Char}(\Lambda,G_c) \eeq


Let $H$ act on $Hom(\Lambda,G_c)$ by 

\beq h(\rho) = \rho \circ C_{\tilde{h}} \quad \forall h \in H \eeq
where $\tilde{h} \in \pi$ is such that $q(\tilde{h}) = h$ and $C_{\tilde{h}}$ is conjugation by $\tilde{h}$. The action is well-defined because $\Lambda$ is abelian. Moreover, the action descends to an action of $H$ on $\text{Char}(\Lambda,G_c)$ in the obvious way.\footnote{Note that since it is an action by an outer conjugation of $\Lambda$, it descends non-trivially to the quotient.} Let $\text{Fix}(H)$ denote the subset of $\text{Char}(\Lambda,G_c)$ consisting of elements fixed by $H$. 

The next lemma states that $r\big(\text{Char}(\pi,G_c)\big) = \text{Fix}(H)$. Hence, the character variety of $Q$ is determined, up to finite fibers, by the action of $H$ on the character variety of the monodromy cover $\T^3$.

\begin{lemma} \label{lemito}
	Suppose $\rho \in \Hom(\Lambda,G_c)$ is such that $\rho = \ol{r}(\tilde{\rho}) = \tilde{\rho}|_{\Lambda}$ for some $\tilde{\rho}$ in $\Hom(\pi,G_c)$. Then $[ \rho ] \in \text{Fix}(H)$.
	
	Conversely, assume $C_{G_c}(\rho(\Lambda)) = 0$ and $[\rho] \in \text{Fix}(H)$. Then $\exists [\tilde{\rho}] \in \text{Char}(\pi,G_c)$ such that $r([\tilde{\rho}]) = [\rho]$.
\end{lemma}

\begin{proof} Let $h \in H$. Then:
	
	\begin{align*} h(\rho)  & = h(\tilde{\rho}|_\Lambda) \\ & = \tilde{\rho}(\tilde{h}) \circ \tilde{\rho}|_\Lambda \circ \tilde{\rho}(\tilde{h})^{-1} \\ & = C_{\tilde{\rho}(\tilde{h})}(\tilde{\rho}|_\Lambda) \\ & = C_{\tilde{\rho}(\tilde{h})}(\rho) \end{align*}
	hence $h[\rho] = [\rho]$, i.e. $[\rho] \in \text{Fix}(H)$.
	
	Conversely, $h[\rho] = [\rho] \implies \rho \circ C_{\tilde{h}} = S_{\tilde{h}}\rho S_{\tilde{h}}^{-1}$ for some $S_{\tilde{h}} \in G_c$. It is easy to see that if $a \in \text{Ker}(q)$, then $S_a^{-1}\rho(a) \in C_{G_c}(\rho(\Lambda))$. Hence $S_{a} = \rho(a)$. Define $\tilde{\rho} : \pi \to G_c$ by $\tilde{\rho}(x) = S_{x}$, $\forall x \in \pi$. Then clearly $\tilde{\rho}|_{\text{Ker}(q)} = \rho$ and if $x, y \in \pi$, the hypothesis on the centralizer implies that $S_{xy} = S_xS_y$, so $\tilde{\rho}(xy) =\tilde{\rho}(x)\tilde{\rho}(y)$. So $\tilde{\rho} \in \Hom(\pi,G_c)$ with $r([\tilde{\rho}]) = [\rho]$.
	
	
	
\end{proof}

In section \ref{exemploexemplar} we will use this result to compute $\text{Char}(\mathcal{G}_6,SL(2,\C))$. We will show that the computation matches with the complexified deformation space of the dual $G_2$-orbifold. This agreement is a non-trivial check of M-theory/IIA duality.


\section{Coassociative ADE Fibrations} \label{adeg2platy}
\subsection{ADE \texorpdfstring{G\textsubscript{2}} --Orbifolds} \label{adeg2orbi}

Fix once and for all an isomorphism $Spin(4) \cong SU(2)_+\times SU(2)_-$. Consider the following data:

\begin{enumerate}
	\item $Q$ a closed, oriented smooth three-manifold with $\pi := \pi_1(Q)$ and a volume form $v_Q$
	
	\item $\Gamma$ a finite subgroup of $SU(2)_-$
	
	\item $p: \mf{V} \to Q$ a rank $4$ oriented real vector bundle with structure group $Spin(4)$ (so endowed with a bundle metric) and such that $p$ is $\Gamma$-invariant: this means there is a further reduction of the structure group of $\mf{V}$ to
	
	\[ G_\mf{V} := C_{Spin(4)}(\Gamma) = SU(2)_+ \times C_{SU(2)_-}(\Gamma) \]
	 where $C_{G}(\Gamma)$ is the centralizer of $\Gamma$ in $G$.
	 
	
	
	
	
	
	
\item A \emph{twisting isomorphism}:

\beq \eta : TQ \stackrel{\cong}{\to} \Lambda^{2,+}\mf{V} \label{twistingiso}  \eeq
	
\end{enumerate}

\begin{remark}	The centralizer $C_{SU(2)}(\Gamma)$ depends almost exclusively on the ADE type of $\Gamma$. Here are the possibilities:
	
	\begin{itemize}
		\item \emph{$\Gamma$ of type $A_1$}: $\Gamma \cong \Z_2$ is the center of $SU(2)$, so:
		
		\beq C_{SU(2)}(\Z_2) = SU(2) \eeq	
	
		\item \emph{$\Gamma$ of type $A_n$, $n \geq 2$}: $\Gamma \cong \Z_n$ and $\Gamma$ lies on a maximal torus $U(1)$ of $SU(2)$. The centralizer is just the torus itself: 
		\beq C_{SU(2)}(\Z_n) =  U(1) \eeq
		
		\item \emph{$\Gamma$ of type $D_n$ for $n > 2$, $E_6$, $E_7$ or $E_8$}: Then: 
		
		\beq C_{SU(2)}(\Gamma) = Z(SU(2)) \cong \Z_2 \eeq

	\end{itemize}

\label{centralizando}
\end{remark}


Condition $3$ implies that $M_0 := \text{tot}(\mf{V})/\Gamma$ inherits a structure of an orbifold bundle:

\beq \overline{p} : M_0  \to Q \eeq

Note also that due to condition $2$, $\eta$ is automatically $\Gamma$-invariant and hence descends to a tensor $\eta_0$ on $M_0$.

\begin{definition} We call $\overline{p} : M_0 \to Q$ an \emph{ADE orbibundle of type $\Gamma$}. We refer to $(M_0,\eta_0)$ as an \emph{ADE $G_2$-orbifold of type $\Gamma$.}
\end{definition}



To justify our choice of terminology, consider the linear model given by $\R^7 = \R^4\oplus\R^3$, with $\R^4$ given its standard flat hyperk\"ahler structure $(\omega_1,\omega_2,\omega_3)$. Let $(x_1,x_2,x_3)$ be  coordinates on $\R^3$ and define:

\[ \eta_{\R^7} = -\sum_{i=1}^3 \omega_i dx_i \in (\R^3)^*\otimes \Lambda^{2,+}(\R^4)^* \]
\[ \mu_{\R^7} =  dx_1dx_2dx_3 \in \Lambda^3(\R^3)^* \]

Then:

\[ \varphi_{\R^7} = \eta_{\R^7} + \mu_{\R^7} \]
is a $G_2$-structure on $\R^7$.

Going back to our definition, once one chooses an Ehresmann connection $\mf{H}$ on $p$, we have a well-defined \emph{vertical cotangent bundle}:

\[ T^*_{\mf{H}}\mf{V} \to \text{tot}(\mf{V}) \]
with the crucial property that its restriction to the zero section satisfies:

\beq T^*_{\mf{H}} \mf{V}|_Q \cong \mf{V} \to Q \eeq
and thus, defining $\Lambda^{2,+}_\mf{H} := \Lambda^{2,+}(T^*_{\mf{H}} \mf{V})$, we can see $\eta$ as:

\beq \eta \in C^\infty(Q,T^*Q\otimes \Lambda^{2,+}_\mf{H})  \eeq
i.e., our $\eta$ is a $3$-form on $\text{tot}(\mf{V})$ which is constant on the fibers of $p$ and such that $\eta$ looks like $\eta_{\R^7}$ at each point. It follows that one can always find a positive function $\lambda : Q \to \R_+$ such that:

\[\varphi = -\eta + \lambda \nu_Q  \]
is a positive form, and hence a $G_2$-structure.

Moreover, it follows from condition $2$ that $\eta$ is $\Gamma$-invariant. Hence $\varphi$ is $\Gamma$-invariant, so it descends to a $G_2$-structure on $M_0$.

Our main goal in this section is to answer the following: 
 
\textbf{Question:} Under which circumstances $(M_0,\eta)$ induces a \emph{closed} $G_2$-structure $\varphi_0$ on $M_0$ such that $\overline{p}$ is a coassociative ADE fibration (i.e., $\varphi_0$ restricts to zero on each fiber)? 

In the next section we will see that the freedom in choosing the Ehresmann connection $\mf{H}$ will be essential in answering this question.

One could also consider a more constrained version of the problem: namely, whether one can construct $\varphi_0$ by choosing $\mf{H}$ to be a \emph{linear} connection on $p$. In order to distinguish this particular setup from the general discussion, we will use the notation $\nabla$ for a linear connection. 

Given $\nabla$, in order to get a well-defined connection on $\overline{p}$, we need to impose: 

\begin{enumerate}

\item[5.] $[\text{Hol}(\nabla),\Gamma]=0 \subset SO(\mf{V})$.

\end{enumerate}

If we have such a setup, then we call $(M_0,\eta,\nabla)$ a \emph{linear ADE $G_2$-orbifold}. The question is then whether one can find $(\nabla, \lambda)$ such that $\varphi_0$ as above is a closed $G_2$-structure on $M_0$ making the fibers coassociative.
	
As we will see, the answer is that the problem can be solved in any bounded neighborhood of the zero-section $Q \subset M_0$, the crucial issue being that the curvature $F_\nabla$ becomes unbounded far from $Q$. In particular, the problem can be solved completely in the special case when $\nabla$ is a \emph{flat connection}. Note that a flat connection $\nabla$ on $\mf{V}$ compatible with $\Gamma$ is given by a linear action of $\pi \times \Gamma$ on $\til{Q}\times \R^4$, where $\Gamma$ acts trivially on $\til{Q}$ the universal cover of $Q$. Equivalently, we have an action of $\pi$ on $\R^4$ commuting with the $\Gamma$-action. Thus, the desired flat connections are in bijection with conjugacy classes in $\Hom(\pi,G_\mf{V})$. 

From another perspective, note that the twisting isomorphism \ref{twistingiso} induces a metric on $Q$, and $\nabla$ is a metric connection on $\mf{V}$. Hence the induced connection $\nabla_Q$ on $Q$ is also metric, and as we will see, the closed condition $d\varphi_0=0$ is related to $\nabla_Q$ being torsion-free. Hence, in order to have closed $G_2$-structures induced from a linear connection on an ADE $G_2$-orbifold, we need our base $Q$ to be a \emph{flat Riemannian $3$-manifold}. We will discuss below some examples of linear $G_2$-orbifolds of type $\Z_n$ over such flat $3$-manifolds.

\begin{remark} \label{normalizador}
More generally, one could relax condition $5$ and ask that $\text{Hol}(\nabla) \leq N_{Spin(4)}(\Gamma)$, the normalizer of $\Gamma$ in $Spin(4)$. This introduces complications that are uninteresting for our purposes: namely, the singularity in $\mf{V}/\Gamma$ acquires nontrivial monodromy under the connection induced by $\nabla$ in the quotient.  The monodromy will be determined by a subgroup of automorphisms of the Dynkin diagram of $\Gamma$, and in that situation, all discussion below must take place on an appropriate finite cover of $Q$.
\end{remark}

Before proceeding with our main question, let's first understand in more practical terms how to build ADE $G_2$-orbifolds. Suppose first that $(Q, G)$ is a closed oriented \emph{Riemannian} $3$-manifold. Fix a spin structure and hence also the spinor bundle $\slashed{S}_Q \to Q$ associated to a principal $SU(2)_+$-bundle $\mathcal{P}_+ \to Q$.

Let $\mathcal{P}_- \to Q$ be any principal $SU(2)$-bundle, and let $\mathcal{W} \to Q$ be the vector bundle associated to $P_-$ via the fundamental representation. Consider the rank $4$ complex vector bundle:

\[ \slashed{S}_Q \otimes \mathcal{W} \to Q \]

This has a real structure coming from $Spin(4) \cong SU(2)_+ \times SU(2)$. So there is a rank $4$ real vector bundle $p: \mf{V} \to Q$ such that:

\[ \slashed{S}_Q \otimes \mathcal{W} \cong \mf{V}\otimes\C \]

The bundle $p: \mf{V} \to Q$ also inherits an action by $SU(2)_+\times SU(2)$, and has the property that

\beq \Lambda^{2,+}(\mf{V}) = TQ \label{identity}\eeq
so we automatically get a twisting map $\eta = Id$. 

Now, if $\mathcal{P}_-$ is chosen so that $\mathcal{W}$ is $\Gamma$-invariant, the same will be true for $\mf{V}$, so we get an orbibundle $\overline{p} : \mf{V}/\Gamma \to Q$. Since $\eta$ is preserved by $\Gamma$, choosing a connection $\mf{H}$ on $\overline{p}$ we get an induced 3-form:

\beq \eta_0 \in \overline{p}^*\Omega^1(Q)\otimes\Omega^{2,+}_{\mf{H}}(M_0) \eeq

We can also work in the linear picture. Assume one is given a linear connection $\nabla_\mathcal{W}$ on $\mathcal{W} \to Q$ with $[\text{Hol}(\nabla_\mathcal{W} ),\Gamma]=0$ and fix the spin connection $\nabla_{LC}$ on $\slashed{S} \to Q$. Then there is an induced connection $\nabla$ on $p: \mf{V} \to Q$ that descends to $\overline{p}$ and again gives us an element:

\beq \eta_0 \in \overline{p}^*\Omega^1(Q)\otimes\Omega^{2,+}_{\nabla}(M_0) \eeq













Thus, if we start with a Riemannian $(Q,G)$, this construction gives us an ADE $G_2$-orbifold $(M_0,\eta_0)$. Conversely, given an ADE $G_2$-orbifold $(M_0,\eta_0)$, the twisting isomorphism \ref{twistingiso} induces a metric $G$ on $Q$ such that the construction above applied to $(Q,G)$ recovers $(M_0,\eta_0)$.

We close this section with a few side remarks.

\begin{remark}
The description in Remark \ref{centralizando} has the following consequences for the structure of $\mathcal{W}$:

\begin{itemize}
	\item If $\Gamma$ is of type $A_1$, any $\mathcal{W}$ is compatible with $\Gamma$. 
	
	If $\Gamma$ is of type $A_n$ for $n \geq 2$, then the structure group reduces to $U(1) \leq SU(2)$ and $\mathcal{W}\cong \mf{L}\oplus \mf{L}^{-1}$, where $\mf{L}$, $\mf{L}^{-1}$ are hermitian line bundle such that $\mf{L}\otimes_{\C}\mf{L}^{-1}$ is trivial.
	
	\item If $\Gamma$ is of type $D_n$ for $n \geq 3$ or of types $E_6$, $E_7$ or $E_8$, then $\mathcal{W} \cong \mf{L}\oplus \mf{L}$, where $\mf{L}$ is a hermitian line bundle such that $\mf{L}^{\otimes 2}$ is trivial.
\end{itemize}
\end{remark}

\begin{remark} Our choice to call $\eta$ a twisting map comes from a  particular case of interest: $\mathcal{P}_- = \mathcal{P}_+$ and $\mathcal{W} = \slashed{S}_Q^* \cong \slashed{S}_Q$. This choice of $\mathcal{W}$ is known in the physics literature as the \emph{partial topological twist} of $7d$ $\mathcal{N}=1$ Super Yang-Mills theory. In that case, one takes $\eta = Id$ as in \ref{identity}.


\end{remark}



\begin{remark}
A consequence of \ref{twistingiso} that will be important later is the following: the bundle $\mf{V} \to Q$ has many compatible complex structures, which are parametrized by the unit sphere bundle $\sph\Lambda^{2,+}_\nabla$. Under $\eta$, this gets identified with the unit tangent bundle $\sph TQ$. Together with the metric induced on $TQ$ by $\eta$, this implies that a choice of $1$-form on $Q$ induces a preferred fiberwise complex structure on $\mf{V}$. 
\end{remark}

\subsection{Donaldson Data}
In the previous section we introduced a general framework for constructing ADE $G_2$-orbifolds. Given such a setup, our main goal now is to understand when such a $G_2$-structure will be closed and vanishing on the fibers. Such a setup will be called a \emph{coassociative ADE fibration}. 

In order to answer this, we will need some results of Donaldson, originally proved in the context of coassociative $K3$ fibrations \cite{donaldson}. These results will also be used in the next section in the proof of our Main Theorem \ref{teoremito} determining which choices of fiberwise smoothings will admit closed $G_2$-structures deforming the original $\varphi$. Modulo details, this subsection is a compressed review of \cite{donaldson}.

Let $p: M^7 \to Q^3$ be a smooth fiber bundle and $T_vM = Ker(dp)$ the vertical tangent bundle. There is an exact sequence:

\beq 0 \to T_vM \to TM \stackrel{dp}{\to} TQ \to 0  \eeq

A connection on $M \to Q$ is equivalent to a section $\iota: TQ \embedd TM$  splitting the sequence; it defines a \emph{horizontal distribution} $\mf{H}  = \iota(TQ) \subset TM$. This induces a splitting of the exterior derivative on $M$ into $d = d_f+d_\mf{H}+F_{\mf{H}}$, where $d_f$ is a fiberwise differential, $d_\mf{H}$ a horizontal differential and $F_{\mf{H}}$ is the curvature operator of $\mf{H}$. We also have an induced \emph{vertical cotangent bundle} $T^*_\mf{H}M \to M$.


\begin{definition} A \emph{hypersymplectic structure} on an oriented four-manifold $(S,v_S)$ is a triple $\underline{\omega} = (\omega_1,\omega_2,\omega_3)$ of symplectic forms such that at each point $p \in S$, $\underline{\omega}_p$ spans a maximal positive-definite subspace $\Lambda^+$ of $\Lambda^2(T^*S)$ with respect to the wedge product.  \label{hypersymplectic}
\end{definition}

In other words, $\omega_i\wedge\omega_j \in \Gamma(X,\text{Sym}^2(T^*S))$ has positive determinant at every point. It is clear that a hypersymplectic structure determines a conformal structure on $S$, namely by declaring $\Lambda^+$ to be the subspace of self-dual forms. By rescaling the volume form $v_S$ one can make $\det (\omega_i\wedge\omega_j) =1$ at all points, so that $G_{ij} := \omega_i\wedge\omega_j$ is a Riemannian metric. This metric will be hyperk\"ahler if and only if $\omega_i\wedge \omega_j$ is a constant multiple of the identity.

Accordingly, we define a \emph{hypersymplectic element} on $(M \to Q,\mf{H})$ to be an element $\eta \in \Gamma(\mf{H}^*\otimes\Lambda^2T^*_\mf{H} M)$ such that at each point $q$, the linear map $\eta_q : \mf{H}_q \to \Lambda^2(T^*_\mf{H} M)_q$ injects $\mf{H}_q$ as a maximal positive subspace with respect to the wedge product. We say $\eta$ is a \emph{hyperk\"ahler element} if in addition, in local coordinates $\ul{x}=(x_1,x_2,x_3)$ on $Q$, we have $\eta = \sum{\omega_i}(\ul{x})dx_i$ with $\omega_i\wedge\omega_j$ positive definite and constant along the fibers.


\begin{proposition} (Donaldson): A closed $G_2$-structure on $(p: M \to Q,\mf{H})$ with coassociative fibers and orientation compatible with those of $M$ and $Q$ is equivalent to a choice of the following data:
	
	\begin{itemize}
		\item A hypersymplectic element $\eta \in \Gamma(M, \mf{H}^*\otimes\Lambda^2T_\mf{H}^* M)$ satisfying: 
		
		\begin{align*}
		d_\mf{H}\eta = 0 \\ 
		d_f\eta =0
		\end{align*}
		
		\item A tensor $\mu \in \Gamma(M, \Lambda^3\mf{H}^*)$ satisfying: 
		
\[ d_f\mu  = -F_{\mf{H}}(\eta) \]
and which is \emph{positive}, i.e., it is the pullback of a volume form on $Q$ multiplied by a positive function on $M$.
	\end{itemize}
	\label{donaldao}
\end{proposition}

	
	

Under these conditions, the $G_2$-structure is given by $\varphi = \eta + \mu$.  We will refer to $(\mf{H},\eta,\mu)$ as \emph{Donaldson data} for a closed $G_2$-structure on $M$.


We denote by $S$ the diffeomorphism type of the fibers of $p$ and suppose we are given a $H^2(S,\Z)$-local system over $Q$. From that we get a flat \emph{affine} bundle $\mathcal{H} \to Q$ with fibers $H^2(S,\R)$. We will say a smooth affine section $h : Q \to \mathcal{H}$ is \emph{positive} if it is an immersion and at every point $x \in Q$, $dh_x(TQ)$ is a positive subspace with respect to the cup product.

\begin{proposition} (Donaldson): Assume $p: M\to Q$ has hypersymplectic fibers and fix a connection $\mf{H}_0$. Let $(x_1,x_2,x_3)$ denote local coordinates on $Q$ on a trivialization of $\mathcal{H}$. Let $h : Q \to \mathcal{H}$ be a positive section and let $\eta \in \Gamma(M, \mf{H}^*_0\otimes\Lambda^2T_{\mf{H}_0}^* M)$ be a hypersymplectic element, locally given by $\sum \omega_i(\ul{x})dx_i$. Suppose that the following local condition holds: 

\[ [ \omega_i ] = \frac{\partial h}{\partial x_i} \]

Then there is a connection $\mf{H}$ on $M \to Q$ such that $d_\mf{H} \eta=0$. 

Furthermore, introducing a radial coordinate $r : S \to [0,\infty)$, assume the following conditions hold: 

\begin{enumerate}
\item $H^1(S) = 0$ 
\item $\eta \sim O(1)$
\item $\mf{H}$ is asymptotically flat on a neighborhood $N \subset M$ of the zero section
\end{enumerate}

Then there is a positive section $\mu \in \Gamma (N,\Lambda^3\mf{H}^*)$ such that $d_f\mu = -F_{\mf{H}}(\eta)$. 

It follows that $(\eta,\mu,\mf{H})$ defines a closed $G_2$-structure on $N$ making $p|_N : N \to Q$ coassociative.
	\label{donaldao2}
\end{proposition}




\textbf{Note:} Propositions \ref{donaldao} and \ref{donaldao2} are stated in \cite{donaldson} for compact fibers. The proof of the first proposition does not rely on compactness so it applies \emph{ipsis litteris} to the non-compact setting. The proof of the second result holds on any open subset of $M$ where $\mf{H}$ is asymptotically flat, since in that case $d_f\mu \sim O(1)$ so $\mu$ can be made positive by multiplication with a positive function $\lambda : Q \to \R_+$.

Let $(h, \eta,\mu,\mf{H})$ be as in \ref{donaldao2}, and let  $\varphi = \eta + \mu$. Consider:

\[ *_\varphi \varphi = *\eta + *\mu \]
where $*\eta \in \Gamma(M, \Lambda^2\mf{H}^*\otimes\Lambda^2T_\mf{H}^* M)$, and $*\mu \in \Gamma(M, \Lambda^4T_\mf{H}^* M)$ can be thought as smoothly varying (w.r.t. $Q$) choice of volume form for the fibers of $p$. 

\begin{proposition} (Donaldson): Given Donaldson data $(\mf{H},\eta,\mu)$ for a closed $G_2$-structure on $(M\to Q,\mf{H})$ with coassociative fibers and $h$ as in \ref{donaldao2}, the $G_2$-structure is \emph{torsion-free} if and only if the following additional conditions are satisfied:

		\begin{align*}
		d_\mf{H}(*\mu) & = 0 \\ 
		d_f(*\eta) & = - F_\mf{H}(*\mu) \\
		d_\mf{H}(*\eta) & = 0
		\end{align*} \label{donaldao3}
\end{proposition}

Consider now a rescaling of the hypersymplectic element $\eta \mapsto \epsilon^{-1}\eta$ for $\epsilon > 0$. Taking the adiabatic limit $\epsilon \to 0$ \cite{donaldson}, the equations for a torsion-free $G_2$-structure become:

		\begin{equation}
		\begin{aligned}[c]
		d_\mf{H}\eta & = 0 \\ 
		d_f\eta & =0 \\
		d_f\mu & =0 
		\end{aligned}
		\; \; \; \; \;
		\begin{aligned}[c]
		d_\mf{H}(*\mu) & = 0 \\ 
		d_f(*\eta) & = 0 \\
		d_\mf{H}(*\eta)  & = 0 \\
		\end{aligned} \label{torcaonula}
		\end{equation}

The results stated so far concern general hypersymplectic fibers. More can be said in the special case of hyperk\"ahler fibers.

\begin{proposition} (Donaldson): If $(\mf{H},\eta,\mu)$ solve the first first five equations in \ref{torcaonula}, then $\eta$ is a hyperk\"ahler element and $\mu$ is the pullback of a volume form on $Q$. Conversely, if $\eta$ is hyperk\"ahler, $h$ is as in \ref{donaldao2} and $\mu = p^*v_Q$, then connections $\mf{H}$ solving the five equations are in bijection with triples of harmonic functions on $Q$.

Moreover, if in addition  $p$ has compact fibers, then 
\[ d_\mf{H}(*\eta) = 0 \] 
is satisfied if and only if locally on $Q$, the image of $h$ is a stationary submanifold w.r.t. the intersection form on $\mathcal{H}$. \label{donaldao4}
\end{proposition}

Again, Proposition \ref{donaldao3} does not rely on compactness. The first part of \ref{donaldao4} also does not depend on it, although when $p$ has compact fibers, Hodge theory guarantees existence and uniqueness of $\mf{H}$. For non-compact hyperk\"ahler fibers such as ALE spaces, one needs to be more careful. In particular, Hodge cohomology for ALE spaces \cite{mazzeo} shows that a connection $\mf{H}$ giving a solution to $d_\mf{H}\eta = 0$ on an ALE fibration cannot be modified via $L^2$-harmonic functions to also solve $d_\mf{H}(*\eta) = 0$.

\subsection{Coassociative ADE Fibrations}

Proposition \ref{donaldao2} motivates the following:

\begin{definition} Let $(\overline{p}: M_0 \to Q,\eta)$ be a ADE $G_2$-orbifold. If there is:
\begin{itemize}
\item A connection $\mf{H}$ on $\overline{p}$ such that $d_\mf{H} \eta = 0$, and 

\item A neighborhood $N$ of the zero section where $\mf{H}$ is asymptotically flat, 
\end{itemize}
then we say that $\overline{p}|_{N} : N \to Q$ is a \emph{coassociative ADE fibration}.
\label{coassade}\end{definition}

It is instructive to understand the conditions for a coassociative ADE fibration with respect to a linear connection. So let $(\overline{p}, \eta,\nabla_0)$ be a linear ADE $G_2$-orbifold. The $G_2$-structure is a positive $3$-form:

\[ \varphi(\nabla_0, \eta,\lambda_0) = \eta + \lambda_0 v_Q \in \Omega^{1,2}_{\nabla_0} \oplus \Omega^{3,0}_{\nabla_0} \]
where $\lambda_0 : Q \to \R_+$ is some positive function.

We want to replace $(\nabla_0, \lambda_0)$ by $(\nabla, \lambda)$ so that $\varphi(\nabla,\eta,\lambda)$ is a \emph{closed} positive $3$-form vanishing on the fibers of $\overline{p}$. 

The condition on $(\nabla,\lambda)$ can be derived as follows: we think of $\eta$ as a twisted one-form $\eta \in \Omega^1(Q, \Lambda^{2,+}_\nabla)$. There is a coupled exterior differential 

\[ d_\nabla : \Omega^1(Q, \Lambda^{2,+}_\nabla) \to \Omega^2(Q, \Lambda^{2,+}_\nabla) \] 
and thus:

\[ d_\nabla \eta \in \Omega^{2,2}_\nabla \]

There is also a curvature map $F_\nabla \in \Omega^2(Q, End(\Lambda^{2,+}_\nabla))$ acting by wedge product on the base and contraction on the fiber. Hence:

\[ F_\nabla(\eta) \in \Omega^{3,1}_\nabla \]

Then it is easy to see that:

\beq d\varphi = d_\nabla\eta + F_\nabla(\eta) + d_f(\lambda)v_Q \eeq
where $d_f(\lambda)v_Q$ is also of type $\Omega^{3,1}_\nabla$. So we have two separate equations:

\[ d_\nabla\eta = 0\]
\[ d_f(\lambda)v_Q = F_\nabla(\eta)  \]

By \ref{donaldao2} above, if the first equation can be solved for $\nabla$, then the second can be solved for a positive $\lambda$ in any neighborhood of $Q$ where $\nabla$ is asymptotically flat. Hence we get a closed $G_2$-structure $\varphi$ making $p$ (and $\overline{p})$ a coassociative fibration.

Note that our $\eta$ is always $O(1)$. Since $\nabla$ is linear, $F_\nabla$ is $O(r)$, so the second equation can be solved for positive $\lambda$ in the interior of any compact neighborhood of $Q$ in $M_0$.


Let us examine the first equation in terms of the spinor construction of $\mf{V}$. Recall that in that situation, $\mf{V}\otimes \C \cong \slashed{S}_Q \otimes \mathcal{W}$ and we have an identification

\[ Id: TQ \stackrel{=}{\to} \Lambda^{2,+}\mf{V} \]

The condition on the metric connection $\nabla$ on $\mf{V}$ is then:

\beq d_\nabla (Id) = 0 \eeq
which just says that $\nabla$ comes from a torsion-free connection on $TQ$. Moreover, if $\mathcal{W} = \slashed{S}_Q^* \cong \slashed{S}_Q$, then since $\nabla$ is metric it must come from the Levi-Civita connection of $Q$.

We conclude that if one chooses the twisting isomorphism $\eta$ to be the partial topological twist of $7d$ $\mathcal{N}=1$ Super Yang-Mills theory (also known as the Chern-Simons twist), and if the connection $\nabla$ is chosen to be induced by $\eta$ from the Levi-Civita connection on $Q$, then there is a closed $G_2$-structure on a (arbitrarily large) neighborhood of $Q$ in $M_0$ making the fibers coassociative. Finally, the coassociative ADE fibration extends to the whole $M_0$ if and only if $Q$ is flat.


In this paper we are mostly interested in coassociative ADE fibrations defined on the whole orbifold $M_0$. Since linear connections are easier to work with, the reader will not be surprised that all examples discussed in the next section involve flat $3$-manifolds. Nevertheless, the framework developed here is flexible enough to allow more interesting constructions - e.g. by working with Ehresmann connections with appropriate decay rates.





\subsection{Examples of Coassociative ADE Fibrations}


\begin{example}
	$\mf{V} = \R^4\times \T^3$ has a standard closed $G_2$-structure:
	
	\beq \varphi = \sum_{i=1}^3 dx_i\wedge \omega_i + dx_{123} \label{g2} \eeq
where $\left\{ x_i \right\}$ are global flat coordinates on $\T^3$ and $\underline{\omega}=(\omega_1,\omega_2,\omega_3)$ is the standard hyperk\"ahler structure on $\R^4$. Here and in what follows, we use the notation $dx_{123} := dx_1\wedge dx_2 \wedge dx_3$. 

	It is easy to check that $\varphi|_{\C^2}=0$ and $\star\varphi|_{\T}=0$, so $\R^4$ is coassociative and $\T^3$ is associative. In fact, this $G_2$-structure is also torsion-free. Its associated metric is just the flat metric, which of course has holonomy $\left\{ 1 \right\} \subset G_2$.
	
	Although this example is rather trivial, it is instructive to understand it in the language we have introduced. We think of $\mf{V}\stackrel{\pi}{\to} \T^3$ endowed with the trivial connection $\nabla = d$ as being associated to the principal $Spin(3)$-bundle $P_{Spin(3)} \to \T^3$ (which is itself flatly trivial) via the standard representation of $Spin(3) \cong SU(2)$. Note that because there is no monodromy, formula \ref{g2} makes sense globally on $\T^3$.  
	
 Moreover, the connection $d$ allows us to construct a vertical cotangent bundle:

\[ T^*_d \mf{V} \to \mf{V} \]
whose sections are $1$-forms that vanish on horizontal vectors. Furthermore, over a fiber $\R^4_x := \pi^{-1}(x)$: 

\beq T^*_d \mf{V}|_{\R^4_x} \cong T^*(\R^4_x) \label{toiah}\eeq

Using the orientation of $\mf{V}$ we also get a vertical bundle of selfdual two-forms:

\[ \Lambda^{2,+}T^*_d\mf{V} \to \mf{V} \]

Because the fibers $\R^4_x$ are coassociative, using \ref{toiah} and the connection $d$ we get identifications:

	\[ \Lambda^{2,+}_dT^*_d\mf{V}|_{\R^4_x} \cong \mathcal{N}_{\R^4_x/\mf{V}} \stackrel{d}{\cong} \pi^*(T\T^3)|_{\R^4_x} \]
		
		It follows that for every $x \in \T^3$ there is an isomorphism of vector spaces
		
		\[ \Lambda^{2,+}T^*_d\mf{V}_{(x,0)} \cong T\T^3_x \]
and hence an isomorphism of vector bundles over $\T^3$:
		
		\beq \Lambda^{2,+}T^*_d\mf{V}|_{\T^3} \cong T\T^3 \label{topatd}\eeq

Conversely, given the isomorphism \ref{topatd} and the connection $d$, one can reconstruct a $3$-form $\varphi$ making the fibers coassociative.	
	
	
	
	
	
	
	
	
	Finally, since $\underline{\omega}$ is a $SU(2)_+$-triple, we can take $\Gamma \leq SU(2)_-$, so that $\underline{\omega}$ is $\Gamma$-invariant. Thus there is a well-defined $G_2$-structure on the quotient $\mf{V}/\Gamma$ with the same properties. 
	
\end{example}

\begin{example} \label{exemplodobob} (\emph{The Hantzche-Wendt $G_2$-orbifold})
	This will be our main example in this paper. It has appeared before as one of the local geometries near a singular stratum in Joyce's construction of compact $G_2$-manifolds \cite{joyce} (see also \cite{acharya}). 
	
	Consider $\mf{V}_0 = \R^4\times_{\K} \G_6$ the total space of a \emph{non-trivial} flat bundle over $\G_6$ induced from the holonomy bundle of $\G_6$ as follows: let $H_{\G_6}$ be the holonomy group of $\G_6$ so that
	
	\[ H_{\G_6} \cong \K = \langle \alpha, \beta ; \alpha^2=\beta^2=1 \rangle \cong \Z_2\times \Z_2 \] 
	
	and consider the action $\rho : \K \to SO(4)$ given by: 

	\[ \rho\alpha(z_1,z_2) = (z_1,-z_2) \]
	
	\[ \rho\beta(z_1,z_2) = (\overline{z_1}, \overline{z_2}) \]
	where we have introduced an auxiliary complex structure $\R^4 \cong \C^2$ and complex coordinates $z_1, z_2$.
	
	Fixing the flat hyperk\"ahler structure on $\R^4$:
	\[ \omega_1 = Im(dz_1\wedge dz_2) \]
	
	\[ \omega_2 = Re(dz_1\wedge dz_2) \]
		
		\[ \omega_3 = \frac{i}{2}\left( dz_1\wedge d\overline{z_1}+dz_2\wedge d\overline{z_2} \right) \] 
	one sees that the induced action of $\K$ on $\Omega^{2,+}(\R^4)$ is given by the following hyperk\"ahler rotations:
	
	\[ \rho\alpha(\omega_1,\omega_2,\omega_3) = (-\omega_1,-\omega_2,\omega_3) \]
	\beq \rho\beta(\omega_1,\omega_2,\omega_3) = (-\omega_1,\omega_2,-\omega_3) \label{abcde} \eeq
	
	Fix a simultaneous trivialization $\mathfrak{U}$ of $T\mathcal{G}_6$ and $\mf{V}_0$. For $U \in \mathfrak{U}$, let $x_1,x_2,x_3$ be locally flat coordinates, and fix orthonormal flat frames $\left\{ dx_1, dx_2, dx_3 \right\}$ for $T^*U$ and $\left\{ \omega_1, \omega_2, \omega_3 \right\}$ for $\Omega^{2,+}(\R^4)\times U$. Consider the problem of extending local forms
	
	\beq \varphi = \sum_{i=1}^3 dx_i\wedge \omega_i + dx_{123} \label{carita}\eeq
	from open subsets of the form $\R^4\times U$ to $\mf{V}_0$.  The monodromy transformations for the $dx_i$'s on local intersections are determined by the action of $H_{\mathcal{G}_6} \cong \K$ given by the matrices $A, B, AB$ in \ref{hwmatrizes}. One sees from \ref{abcde} that the $\omega_i$'s transform according to the same matrices. Since these matrices are symmetric, we conclude that the element:
	
	\beq \eta := \sum_{i=1}^3 dx_i\wedge \omega_i \eeq
	glues to a global flat section. Obviously $dx_{123}$ also glues to a global volume form on $\G_6$, so together they give a well-defined $G_2$-structure. This $G_2$-structure is closed and torsion-free; the associated metric has holonomy $\K \subset G_2$.

Now let $\Gamma \cong \Z_2 \leq SU(2)_-$ act on $\C^2$ in the natural way. It is easy to see that this action is compatible with $\rho$: this means that the monodromy representation of $\mf{V}$ is an element of $\Hom(\pi_1(\G_6),C_{SU(2)}(\Z_2))$, which is clear since $C_{SU(2)}(\Z_2) = SU(2)$. It follows that \ref{carita} descends to a $G_2$-structure with coassociative fibers on 

\[ N_0 := \mf{V}_0/\Z_2 = \C^2/\Z_2 \times_\K \G_6 \] 

We call $N_0$ the Hantzsche-Wendt $G_2$-orbifold.

There are three different smoothings of $N_0$ with compatible $\K$-actions, and Joyce \cite{joyce} extends the $G_2$-structure over them. The associated metric on these smoothings has holonomy $SU(2)\rtimes\K \subset G_2$.

	In this example we also have a bundle of vertical selfdual two-forms:
	
	\[ \Omega^{2,+}_\nabla(\mf{V}) \to \G_6 \]
	which is now a locally constant sheaf with local sections $(\underline{\omega})$. The discussion from the previous example generalizes to this setting, since coassociativity is a pointwise condition on $\G_6$. Hence we have an isomorphism:
	
	\[ \eta : \Lambda^{2,+}_\nabla \cong T\G_6 \]
	and conversely, given the isomorphism and $\nabla$, we can reconstruct $\varphi$.
	
	Note that if instead we had $\Gamma =\Z_n$, then $C_{SU(2)}(\Z_n) = U(1)$ does not contain $\K$. In this situation the singularity $\C^2/\Z_n$ acquires non-trivial monodromy dictated by $[\K,\Z_n]$. The issue here is that $\K$ is the image in $SO(4)$ of a $\K^2$ subgroup of $Spin(4)$ which is not contained in $C_{Spin(4)}(\Z_n) = SU(2)_+\times U(1)$, but is contained in $N_{Spin(4)}(\Z_n) = SU(2)_+\times U(1)^2$ - c.f. Remark \ref{normalizador}. We will explain in Example \ref{meuexemplo} how to fix this issue.
	
	In the next section we will study smoothings of ADE $G_2$-orbifolds via unfolding of singularities and how it connects to Riemannian Higgs bundles \ref{pwequations}. This will give us two new ways of computing the deformation space: one in terms of solutions to \ref{pwequations} and one in terms of the complex flat connection $A+i\theta$. We will check that both computations agree with Joyce's for $N_0$ in Section \ref{exemploexemplar}.
\end{example}

\begin{example} 

Let the setup be as in the previous example, except that we modify the value of $\rho$ on $\beta \in \K$ to be:

\[ \rho\beta(z_1,z_2)=(-\overline{z}_2,\overline{z}_1) \]

Then everything proceeds exactly as before, with an added bonus: now the action of $\Z_n \leq SU(2)_-$ commutes with $\rho$, so one gets a structure of coassociative ADE fibration on the quotient $\mf{V}/\Z_n$.

\end{example}

\begin{example} \label{meuexemplo} 
We now illustrate how to modify a given example via an affine isometry on the base.

The basic point is that the short exact sequence \ref{exactbieber} is not unique; one can modify the lattice $\Lambda$, for example by modifying its period along one direction, as long as one modifies the group $H_\pi$ accordingly.

To illustrate this point, start with the exact sequence for $\G_6$:
	
	\beq 1 \to \Z^3 \to \pi_1(\G_6) \to \K \to 1 \eeq
	
	Bieberbach's first theorem says that $\G_6$ is a quotient of the three-torus $\T$. This is realized via the following (free) action of $\K = \langle \alpha, \beta \rangle$ on $\T$:	
	
	\begin{align*}	\alpha(x_1,x_2,x_3) & = (-x_1+\frac{1}{2},-x_2,x_3+\frac{1}{2}) \\
	\beta(x_1,x_2,x_3) & = (-x_1, x_2+\frac{1}{2}, -x_3)
	\end{align*}
	
	So $\G_6 = \T/\K$. However, a second possible description is as follows: consider the lattice $L = 2\Z\oplus \Z\oplus\Z$ and let $\T' = \R^3/L$ be a torus isogenous to $\T$. Then $\G_6 = \T'/D_8$, where $D_8 \cong \Z_2 \rtimes \K$ is the dihedral group with $8$ elements. 
	
	Let $\alpha', \beta'$ be two generators of $D_8$ satisfying $(\alpha')^2=1$ and $(\beta')^4=1$. The action is given by:
	
	\begin{align*}	 	\alpha'(x_1,x_2,x_3) & = (-x_1+\frac{3}{4},-x_2,x_3+\frac{1}{2}) \\
	\beta'(x_1,x_2,x_3) & = (-x_1+\frac{1}{4}, x_2+\frac{1}{4}, -x_3)
	\end{align*}
	
	This provides a second short exact sequence for $\pi_1(\G_6)$:
	
	\beq 1 \to L \to \pi_1(\G_6) \to D_8 \to 1 \eeq
	
	Note that $(\beta')^2(x_1,x_2,x_3) = (x_1,x_2+\frac{1}{2},x_3)$ is a translation by an order 2 element generating the center $\Z_2 \leq D_8$, and as such it does not contribute to the holonomy.
	

Now consider the following action $\rho' : D_8 \to SO(4)$

\[ \rho'\alpha'(z_1,z_2) = (-z_1,z_2) \]
\beq \rho'\beta'(z_1,z_2) = (-i\overline{z_2}, i\overline{z_1}) \eeq

This action commutes with the action of $\Z_n \leq SU(2)_-$ for all $n$. Hence we can construct nontrivial flat $A_n$ data using this action. 

A quick computation shows that $\rho'(\beta')^2$ acts trivially on the hyperk\"ahler triple, and that $\rho'$ maps $D_8/\langle (\beta')^2 \rangle \cong \K = H_{\G_6}$ to the action \ref{abcde} above. Hence the quotient $\C^2/\Z_n\times_{D_8} \G_6$ inherits a structure of coassociative ADE fibration.
	
	The key feature in this example is that there is a central $\Z_2$-symmetry given by $\langle \rho'(\beta')^2 \rangle$, and for this reason the preimage of $\rho'(D_8)$ in $Spin(4)$ is another copy of $D_8$ sitting inside $C_{Spin(4)}(\Z_n)$. 
	
	Since $\rho'(\beta')^2$ acts trivially on the hyperk\"ahler triple, this symmetry is not visible at the level of the $G_2$-form, but as a non-geometric symmetry is must be remembered when applying the string duality explored in this paper. In physics jargon, $\rho'(\beta')^2$ gives rise to a \emph{B-field} on the dual Calabi-Yau space. Mathematically, the dual Calabi-Yau will inherit a flat $\Z_2$-gerbe. So in a precise sense, a nontrivial affine isometry gives rise to a nontrivial flat gerbe. 
	
	This example suggests that many ADE orbifolds with $G_2$-structures will carry non-trivial gerbes. Since our focus in this paper is on the $G_2$-structure itself, we will ignore the gerby structure and leave for future work to determine how it affects our constructions below.
\end{example}

\begin{example}
There is a simpler version of the previous example involving $Q = \G_2$ and $\Gamma \cong \Z_n$. One takes $\Z_2\times \Z_2$ generated by $\overline{\alpha}, \overline{\beta}$ and an action $\overline{\rho} : \Z_2\times \Z_2 \to SO(4)$ by:

\[ \overline{\rho}\overline{\alpha}(z_1,z_2) = (-z_1,z_2) \]
\[ \overline{\rho}\overline{\beta}(z_1,z_2) = (-z_1,-z_2) \]

This action commutes with $\Z_n$ and $\overline{\rho}\overline{\beta}$ acts trivially on the hyperk\"ahler triple $\underline{\omega}$. The generator $[\overline{\alpha}]$ of the quotient $\Z_2\times\Z_2/\langle \overline{\beta}\rangle \cong \Z_2$ acts as $\overline{\alpha}(\underline{\omega}) = (\omega_1,-\omega_2,-\omega_3)$ which mirrors the action of $H_{\G_2}$ on locally flat $1$-forms of $\G_2$. Hence there is a well-defined $G_2$-structure on the orbifold $\C^2/\Z_n\times_{\Z_2} \G_2$.

\end{example}

\section{A Deformation Family of Coassociative ALE Fibrations} \label{deformation}

\subsection{Kronheimer's Construction}

Recall that we denote by $\mathfrak{g}_c$ a simple complex Lie algebra McKay dual to a finite ADE group $\Gamma \leq SU(2)$, $\mathfrak{h}_c$ a Cartan subalgebra, and $W$ the Weyl group. The compact real form of $\mathfrak{g}_c$ is denoted $\mathfrak{g}$, and $\mathfrak{h}$ is the associated real Cartan subalgebra.

We start by reviewing the construction of the Brieskorn-Grothendieck versal deformation of the quotient singularity $\C^2/\Gamma$: let $x$ be a subregular nilpotent element of $\mathfrak{g}_c$ and complete it to a $\mathfrak{sl}(2,\C)$-triple $(x,h,y)$. Define the \emph{Slodowy slice}:

\beq \mathcal{S} = x + \mathfrak{z}_{\mathfrak{g}_c}(y) \subset \mathfrak{g}_c \eeq
where $\mathfrak{z}_{\mathfrak{g}_c}(y)$ is the centralizer of $y$, i.e. the kernel of the adjoint action of $G_c$ on $y$.

Consider the GIT adjoint quotient $\mathfrak{g}_c \to \mathfrak{g}_c//G_c$. Chevalley's theorem says that $\C[\mathfrak{g}_c]^{G_c} \cong \C[\mathfrak{h}_c]^W$, so $\mathfrak{g}_c//G_c \cong \mathfrak{h}_c/W$. Define $\Psi : \mathcal{S} \to \mathfrak{h}_c/W$ to be the restriction of $\mathfrak{g}_c \to \mathfrak{h}_c/W$ to $\mathcal{S}$.  

\begin{theorem} (Slodowy \cite{slodowy}): The family $\Psi$ has the following properties:
	\begin{enumerate}
		\item $\Psi$ is a flat, surjective holomorphic map
		\item $\Psi^{-1}(0) \cong \C^2/\Gamma$ 
		\item Given any other map $\Psi' : \mathcal{A} \to \mathcal{B}$ satisfying properties $1$ and $2$ there is a map $\beta : (\mathcal{B},b) \to (\mathfrak{h}_c/W,0)$ such that $\Psi' = \beta^*\Psi$. The map $\beta$ is not uniquely determined, but its derivative $d\beta_b$ is unique.\footnote{This uniqueness at the infinitesimal level is known as \emph{miniversality}. Any two miniversal deformations of an ADE singularity are isomorphic, and their reduced Kodaira-Spencer map is an isomorphism.}
		\item $\Psi$ is equivariant with respect to natural $\C^*$-actions on $\mathcal{S}$ and $\mathfrak{h}_c/W$
	\end{enumerate}
\end{theorem}

In other words, $\Psi$ is the \emph{Brieskorn-Grothendieck $\C^*$-miniversal deformation of $\C^2/\Gamma$}. Thus, the Slodowy slice is a geometric realization of the deformations of $\C^2/\Gamma$ inside the Lie algebra $\mathfrak{g}_c$. 

This embedding of $\mathcal{S}$ into $\mathfrak{g}_c$ comes with a symmetry group. Let $Z= Z_{G_c}(x)\cap Z_{G_c}(y)$ be the \emph{reductive centralizer} of $x$ with respect to $h$.\footnote{The name is due to the fact that the identity component $Z^0$ is reductive, and the component groups $Z_{G_c}(x)/Z^0_{G_c}(x)$ and $Z/Z^0$ coincide.} Its action on $\mathcal{S}$ commutes with $\C^*$, so there is an action of $\C^*\times Z$ on $\mathcal{S}$. The action of $Z$ restricts to act on the fibers of $\Psi$ (i.e., $\Psi$ is $Z$-invariant). The group $\C^*\times Z$ is called the symmetry group of the Slodowy slice.

\begin{lemma}
	$Z \cong \C^*$ for $\mathfrak{g}_c$ of type $A_n$, and $Z = \left\{ e \right\}$ for types $D_n$ and $E_{6,7,8}$.
\end{lemma}

\begin{proof}
	See Slodowy's book \cite{slodowy}.
\end{proof}

Kronheimer \cite{kronheimer1} identified the moduli of ALE hyperk\"ahler $4$-manifolds with smoothings of $\C^2/\Gamma$. The construction goes as follows: first one shows that for every $\Gamma$ there is a flat simply-connected hyperk\"ahler manifold $(Y_\Gamma,I,J,K)$ and a compact Lie group $F$ acting on $Y_\Gamma$ by hyperk\"ahler isometries. Let $\mathfrak{f}$ denote the Lie algebra of $F$ and $\mathfrak{z}$ its center. Then one constructs a \emph{hyperk\"ahler moment map}:

\beq \mu = (\mu_1,\mu_2,\mu_3) : Y_\Gamma \to \R^3\otimes \mathfrak{f}^* \eeq
such that for $\xi \in \R^3\otimes \mathfrak{z}^*$ the \emph{hyperk\"ahler quotient}:

\beq S_\xi  = \mu^{-1}(\xi)/F \eeq
is well-defined and is a (possibly singular) hyperk\"ahler space. In particular, it has a hyperk\"ahler triple $(I_\xi,J_\xi,K_\xi)$ induced from $Y_\Gamma$.

A crucial property of the construction is that there is an isomorphism:
\[ \mathfrak{z}^* \cong \mathfrak{h} \]

Using this identification, one can then prove that $\forall \xi \in \R^3\otimes\mathfrak{h}$,  the hyperk\"ahler space $S_\xi$ has a compatible ALE metric, and that $S_\xi$ is non-singular if and only if $\xi \notin \R^3\otimes\Delta$, where $\Delta := \bigcup_{\upsilon} \Pi_\upsilon \subset \mathfrak{h}$ is the union of hyperplanes $\Pi_\upsilon$ orthogonal to a root $\upsilon$.

The relation between Kronheimer's construction and $\Psi$ goes as follows: fix an identification $\R^3 = \R \oplus \C$ and consider the \emph{complexified moment map} $\mu_c := \mu_2+i\mu_3 : Y_\Gamma \to \C\otimes \mathfrak{f}^*$. Then, for $\chi_2+i\chi_3 \in \C\otimes \mathfrak{h} = \mathfrak{h}_c$: 

\beq S_{(0,\chi_2,\chi_3)} = \big(\mu_1^{-1}(0)\cap\mu_c^{-1}(\chi_2+i\chi_3)\big)/G \eeq
is an affine variety with respect to the complex structure $I_{(0,\chi_2,\chi_3)}$. After passing to a normalization, these spaces fit into the \emph{Kronheimer deformation family:}

\beq \Theta_0 : \mathcal{K}_0 \to \C\otimes\mathfrak{z}^* \cong \mathfrak{h}_c \eeq
which is a surjective flat holomorphic map with $\Theta^{-1}_0(0) \cong \C^2/\Gamma$.

In this discussion, there is a preferred generator $\omega_c$ of the relative canonical bundle of $\mathcal{K}_0$. It is induced by the $I$-holomorphic symplectic form $\omega_J + I\omega_K$ on $Y_\Gamma$ and it satisfies $[\omega_c|_{\Theta_0^{-1}(\chi_2+i\chi_3)}] = \chi_2+i\chi_3$. Conversely, a choice of generator $\omega_c$ induces a period map:

\[ \mathscr{P} : \C\otimes\mathfrak{z}^* \stackrel{\cong}{\to} \mathfrak{h}_c \]
sending $\chi_c$ to $[\omega_c|_{\Theta_0^{-1}(\chi_c)}]$.

Recall that Slodowy's deformation family $\Psi : \mathcal{S} \to \mathfrak{h}_c/W$ is versal for $\C^2/\Gamma$, so $\Theta_0$ must be induced from it by pullback from a map between the parameter spaces. Kronheimer proved that $\Theta_0$ is equivariant with respect to a $\C^*$-action on $\mathcal{K}_0$ and weight $2$ dilations on $\mathfrak{h}_c$. Due to Looijenga's description of the period map for $\Psi$ \cite{looijenga}, it follows that $\Theta_0$ is induced from $\Psi$ via pullback by the projection map $p_W : \mathfrak{h}_c \to \mathfrak{h}_c/W$. 

\begin{definition} Fix non-zero $\chi, \chi_2, \chi_3 \in \mathfrak{h}$. We say $\xi=(\chi_1,\chi_2,\chi_3) \in \R^3\otimes\mathfrak{h}$ is \emph{generic} if $\xi \notin \R^3\otimes \Pi_\upsilon$ for any root $\upsilon$.
\end{definition}

If $\xi$ is generic, the space $S_\xi$ is a nonsingular hyperk\"ahler manifold, and there is a \emph{resolution of singularities} $r _\chi: S_\xi \to S_{(0,\chi_2,\chi_3)}$. Tjurina \cite{tjurina}, building on previous work of Brieskorn \cite{brieskorn}, proved that a flat holomorphic map $f : R \to B$ with two-dimensional fibers admitting at most finitely many rational double points admits a local resolution of singularities: around any point $b \in B$ there is an open set $U \subset B$ such that the family $f|_{f^{-1}(U)}$ admits a simultaneous resolution of all fibers, i.e. a commutative diagram whose maps restricted to the fibers are resolutions of singularities. Kronheimer's construction gives the simultaneous resolution for the Brieskorn-Grothendieck $\C^*$-miniversal deformation. 

Since any choice of $\chi \in \mathfrak{h} \setminus \Delta$ makes $\xi$ generic, it induces a \emph{simultaneous resolution} ${\Theta}_{\chi} : {\mathcal{K}}_\chi \to \mathfrak{h}_c$ of $\Theta_0$ (i.e., ${\Theta}_\xi = \Theta_0 \circ r_\xi$).

We summarize this discussion in the following:

\begin{theorem} (Brieskorn, Kronheimer, Slodowy, Tjurina): For every $\chi \in \mathfrak{h} \setminus \Delta$, there is a commutative diagram:
	
	\[ \begin{tikzcd}
	{\mathcal{K}_\chi} \arrow[swap]{r}{r_\chi}   \arrow[swap]{d}{{\Theta}_\chi}  & \mathcal{K}_0 \arrow[swap]{d}{\Theta_0} \arrow[rightarrow]{r}{} &  \mathcal{S} \arrow[swap]{d}{\Psi} \\   																				
	\mathfrak{h}_c \arrow[rightarrow]{r}{=} 										  &             \mathfrak{h}_c \arrow[rightarrow]{r}{p_W}   		   & \mathfrak{h}_c/W 		 
	\end{tikzcd}
	\]
	satisfying the following properties:
	
	\begin{enumerate}
		\item ${\Theta}_\chi$ is a flat, surjective holomorphic map, with fibers diffeomorphic to the minimal resolution $\til{\C^2/\Gamma}$ of $\C^2/\Gamma$ and admitting an ALE hyperk\"ahler structure
		\item ${\Theta}_\chi$ is a simultaneous resolution of $\Theta_0$, i.e., $r_\chi|_{S_{\xi}}$ is a resolution of singularities of $S_{(0,\chi_2,\chi_3)}$
		\item ${\mathcal{K}}_\chi$ inherits a $\C^*$-action from $Y^k$ such that ${\Theta}_\chi$ is $\C^*$-equivariant
	\end{enumerate} \label{krony}
\end{theorem}

Moreover, one can also prove:

\begin{theorem} (Kronheimer \cite{kronheimer2}): Given a (smooth) hyperk\"ahler ALE space $S$, there is a generic $\xi = (\chi,\chi_2,\chi_3)$ such that $S \cong S_\xi$ as hyperk\"ahler manifolds.
	\label{kronaaaldo}
\end{theorem}

\subsection{The Twistor Family of ALE Hyperk\"ahler Structures}

One should think of the base $\mathfrak{h}_c$ of the Kronheimer family ${\mathcal{K}}_\chi$ as parametrizing \emph{infinitesimal deformations of the holomorphic symplectic structure on $\til{\C^2/\Gamma}$}. The reason is the following: let $\mathfrak{h}_{>}$ be the \emph{positive Weyl chamber} in $\mathfrak{h}$. Via the McKay correspondence, $\mathfrak{h}_{>}$ is identified with the cone of K\"ahler classes of $\til{\C^2/\Gamma}$, with tangent spaces $\mathfrak{h}$. A choice of complex structure on $\til{\C^2/\Gamma}$ induces an isomorphism $T(\mathfrak{h}_{>})\otimes\C \cong \mathfrak{h}_c$, so the deformation parameter is a complexified K\"ahler class, which is in fact holomorphic \cite{hitchinkarlhedelindstromrocek}

For our purposes, we need to make a clear distinction between deformations of a \emph{holomorphic symplectic structure} (HS) and deformations of a \emph{hyperk\"ahler structure} (HK). The main point is that, even though Kronheimer's construction produces all HK ALE spaces, it does not fit them \emph{all} together in a family induced from the Slodowy slice. In order to write diagram in Theorem \ref{krony}, one needs to fix a complex structure $I$ and an element $\chi \in \mathfrak{h}$. This fixes the HK-structure but does not account for HK deformations coming from resolving $\widetilde{\C^2/\Gamma}$. One way to see this is by noting that the minimal resolution contains projective curves, while deformations are affine varieties. However, for our purposes we will need to work with the full HK family, so one needs to introduce the family of all \emph{smoothings} of $\C^2/\Gamma$ - deformations \emph{and} resolutions. As we will see, in doing so one loses the holomorphic nature of the construction, which is in fact a good thing - there is no preferred complex structure in a HK deformation.

Let $Im\mathbb{H} = \R^3$ be the adjoint representation of $SU(2)_+$. In comparing ${\Theta}_{\chi'}$ and ${\Theta}_{\chi}$, they correspond to different choices of K\"ahler classes for a fixed complex structure inducing a linear isomorphism $\R^3 \cong \R\oplus\C$. In other words, the complex structure $I$ is selected on every fiber once a choice of  splitting $\text{Im}\mathbb{H} \cong \R\oplus\C$ has been made, or equivalently, a choice of $U(1) \subset SU(2)_+$. We write $\mathfrak{h}_{\mathbb{H}} := \mathfrak{h}\otimes \text{Im}\mathbb{H}$.

Under the McKay identification $\mathfrak{h} \cong H^2(S_\xi,\R)$, for every fixed $\xi=(\chi,\chi_2,\chi_3)$, one should think of the deformation parameter: 

\beq \chi_2+i\chi_3  \in \mathfrak{h}_c \setminus  \Delta_\C\eeq
as a choice of cohomology class for a $I_{(\chi,\chi_2,\chi_3)}$-holomorphic symplectic form on the fiber ${\Theta}^{-1}_{\chi}(\chi_2+i\chi_3) \cong S_\xi$. 

This is where the distinction between the HS and HK structures on the fibers comes in. For each $\chi \in \mathfrak{h}$, the family ${\Theta}_{\chi}$ parametrizes HS-deformations of $S_{(\chi,0,0)}$: there is a relative two-form $\omega_c\in \Omega^2_{\Theta_\chi}({\mathcal{K}}_\chi)$ that restricts to a a holomorphic symplectic form $\omega_c^{\chi_2+i\chi_3}$ on every fiber $\Theta_{\chi}^{-1}(\chi_2+i\chi_3) = S_{(\chi,\chi_2,\chi_3)}$, varying holomorphically with $\chi_2+i\chi_3 \in \mathfrak{h}_c$. It is clear that for $\chi' \neq \chi$, the manifolds $S_{(\chi',\chi_2,\chi_3)}$ and $S_{(\chi,\chi_2,\chi_3)}$ are isomorphic as holomorphic symplectic manifolds.

However, because $\chi' \neq \chi$, they define distinct hyperk\"ahler manifolds. The difference can be illustrated in the compact case by the following proposition, which is a consequence of Yau's solution to the Calabi conjecture:

\begin{proposition}
	Let $(S,\Omega)$ be a compact holomorphic symplectic manifold with respect to a complex structure $I$ and $\alpha \in H^{1,1}(S)$ a K\"ahler class. Then there is a unique hyperk\"ahler structure $(I,J,K)$ on $S$ such that $[\omega_I] = \alpha$ and $\Omega = \omega_J+i\omega_K$.
\end{proposition}

For the case when $(S,\Omega)$ is a holomorphic symplectic ALE space, by theorems \ref{krony}, \ref{kronaaaldo}, $(S,\Omega,\alpha)$ can be seen as a Kronheimer fiber $S_{(\alpha,[\Omega])}$ and hence there is a unique compatible HK-structure.

The deformation space of HK-structures will be modeled on a family of smoothings of $\C^2/\Gamma$. The key point is that the base $\mathfrak{h}_c$ must be replaced by a larger space that also includes the twistor sphere of $\widetilde{\C^2}/\Gamma$

In order to do so, we need a result from \cite{kronheimer2}:

\begin{proposition} (Kronheimer): Let $X = \widetilde{\C^2/\Gamma}$ be a hyperk\"ahler ALE space and $Z \to \C\p^1$ its twistor family. There is a $\C^*$-deformation of $\C^2/\Gamma$:

\[ \Phi : \mathscr{Y} \to \C^2 \]
fitting into a commutative diagram:

	\[ \begin{tikzcd}
	Z \arrow[swap]{r}{r}   \arrow[swap]{d}  & \bigslant{\mathscr{Y}\setminus \Phi^{-1}(0)}{\C^*} \arrow[swap]{d}{\overline{\Phi}} \arrow[leftarrow]{r}{} &  \mathscr{Y} \arrow[swap]{d}{\Phi} \\   																				
	\C\p^1 \arrow[rightarrow]{r}{=} 										  &             \C\p^1 \arrow[leftarrow]{r}   		   & \C^2 		 
	\end{tikzcd}
	\]
where the right square is obtained by dividing by $\C^*$ and the map $r$ is a simultaneous resolution of $\overline{\Phi}$ inducing the minimal resolution on every fiber.
\end{proposition}

Since $\Phi$ is a $\C^*$-deformation of $\C^2/\Gamma$, it is induced from $\Psi$ via pullback by a map:

\[ \zeta : \C^2 \to \mathfrak{h}_c/W \]

Together with $p_W$, we can form the fiber product:

	\[ \begin{tikzcd}
	\C^2 \times_{\mathfrak{h}_c/W} \mathfrak{h}_c \arrow[swap]{r}   \arrow[swap]{d}  & \C^2 \arrow[swap]{d}{\zeta}  \\   																				
	\mathfrak{h}_c \arrow[rightarrow]{r}{p_W} 										  &             \mathfrak{h}_c/W 		 
	\end{tikzcd}
	\]

Let:

\[ \mathcal{R} := \mathscr{Y}\times_{\mathcal{S}} \mathcal{K}_0 \]

Since we also have that $\mathscr{Y}=\zeta^*\mathcal{S}$ and $\mathcal{K}_0 = p_W^*\mathcal{S}$, general considerations imply there is a unique map:

\[ \widetilde{\Upsilon} : \mathcal{R} \to \C^2 \times_{\mathfrak{h}_c/W} \mathfrak{h}_c \]
making all the obvious diagrams commutative.

Recall that the Kronheimer family $\Theta_0$ had an implicit complex structure (which we denoted by $I$) built into it, so that the fibers $\Theta_0^{-1}(\chi_2+i\chi_3) = S_{(0,\chi_2,\chi_3)}$ are holomorphic symplectic with respect to $I$. In contrast, the fiber of $\widetilde{\Upsilon}$ over $(a,\chi_2+i\chi_3) \in \C^2 \times_{\mathfrak{h}_c/W} \mathfrak{h}_c$ consists of $S_{(0,\chi_2,\chi_3)}$ obtained from the Kronheimer construction with complex structure $[a] \in \C\p^1$.

In fact, for our purposes it is more convenient to work with the composition of $\widetilde{\Theta}$ with the projectivization map $\C^2 \to \C\p^1$, which we denote by

\beq \Upsilon : \mathcal{R} \to \C\p^1 \times_{\mathfrak{h}_c/W} \mathfrak{h}_c \eeq

Now, any choice of $\chi \in \mathfrak{h} \setminus \Delta$ allows us to build:

\[ \widetilde{\Upsilon}_\chi : \mathcal{R}_\chi := \mathscr{Y}\times_{\mathcal{S}} \mathcal{K}_\chi \to \C^2 \times_{\mathfrak{h}_c/W} \mathfrak{h}_c \]
and

\beq \Upsilon_\chi : \mathcal{R}_\chi \to \C\p^1 \times_{\mathfrak{h}_c/W} \mathfrak{h}_c \label{upsila}\eeq
whose fiber over $([a], \chi_2+i\chi_3)$ is the holomorphic symplectic space $S_{(\chi,\chi_2,\chi_3)}$ in complex structure $[a]$ which, whenever $(\chi,\chi_2,\chi_3)$ is generic, is endowed with a natural K\"ahler form $\omega_\chi$ such that $[\omega_\chi] = \chi$. The relative two-form $\omega_c$ in $\mathcal{K}_\chi$ defines a global holomorphic section on $\mathcal{R}_\chi$:

\beq \omega_c \in \Omega^2_{\Upsilon_\chi}(\mathcal{R}_\chi)\otimes \mathcal{O}(2) \eeq
such that $(\omega_\chi,\omega_c)$ is a hyperk\"ahler triple over the generic locus on the base.

We summarize this discussion in the following:

\begin{proposition} The map $\widetilde{\Upsilon}_\chi : {\mathcal{R}}_\chi  \to \C^2 \times_{\mathfrak{h}_c/W} \mathfrak{h}_c$ defines a family of hyperk\"ahler ALE spaces over the generic locus on the base. The map fits into a commutative diagram:
	
	\[ \begin{tikzcd} \label{rody}
	{\mathcal{R}}_\chi \arrow[swap]{r}   \arrow[swap]{d}{{\widetilde{\Upsilon}_\chi}} 				 &   \mathcal{S} \arrow[swap]{d}{\Psi} \\   																				
	\C^2 \times_{\mathfrak{h}_c/W} \mathfrak{h}_c \arrow[rightarrow]{r} 			   	   &            \mathfrak{h}_c /W 		 
	\end{tikzcd}
	\]

Each fiber $(\widetilde{\Upsilon}_\chi)^{-1}(a,\chi_2+i\chi_3)$ is a deformation of the complex space $(\widetilde{\Psi^{-1}(0)},[a])$, and the generic fiber has a hyperk\"ahler triple induced from a relative triple $\ul{\omega}_{\text{unf}} = (\omega_\chi,\omega_c)$ varying smoothly over the generic locus in $\C^2 \times_{\mathfrak{h}_c/W} \mathfrak{h}_c$. 

\label{rodzao}
\end{proposition}

The notation $\ul{\omega}_{\text{unf}}$ is meant to emphasize that this element induces HK-structures on the \emph{unfoldings} $S_\xi$ of the singularity $\C^2/\Gamma$.

\subsection{Cameral Sections} \label{cameralsections}

Let $(p_0: M_0 \to Q,\nabla_0,\eta_0)$ be a coassociative ADE fibration of type $\Gamma$ with closed $G_2$-structure $\varphi_0$. We write $(\eta_0,\mu_0,\mf{H}_0)$ for its associated Donaldson data; in particular, $\mf{H}_0$ is the horizontal distribution associated to the connection $\nabla_0$ on $\mf{V}$. Let again $\mathfrak{g}$ be the compact Lie algebra associated to $\Gamma$, $\mathfrak{h}$ a Cartan subalgebra, $W$ the Weyl group and $r := \text{rank}(\mathfrak{g}) = \dim(\mathfrak{h})$. 

Our first goal is to use the cocycle $\zeta_\mf{V}$ of $\mf{V}$ to build a family version of the map ${\Upsilon}_\chi$ over $Q$. This consists of:
\begin{enumerate}
\item A complex fiber bundle $t: \mathcal{E}_c \to Q$ whose fibers can be naturally identified with the HK deformation spaces of the fibers of $p_0$. It is given by:

\[ \mathcal{E}_c := \sph\Lambda^{2,+}_\nabla \otimes \mathfrak{H}_c \]
where $\mathfrak{H}_c \to Q$ is a rank $r$ vector bundle with fibers $\mathfrak{h}_c$ endowed with a fiberwise $W$-action, to be defined in the proof of Lemma \ref{bundleball} (see also the next remark). There is also another fiber bundle:

\[ t_W : \mathcal{E}_c/W \to Q \]
and a natural projection $e : \mathcal{E}_c \to \mathcal{E}_c/W$.


\item A family of (possibly singular) hyperk\"ahler $4$-manifolds $u : \mathcal{U} \to \mathcal{E}_c$ with a structure of smooth fiber bundle on a dense open $\mathcal{U}^\circ \subset \mathcal{U}$.
\end{enumerate}



In this picture, the twistor $\C\p^1$ appearing in \ref{upsila} is identified with the unit spheres in the fibers of $\sph\Lambda^{2,+}_\nabla$ via the isomorphism $\p(\slashed{S}^+_\mf{V}) \cong \sph\Lambda^{2,+}(\mf{V})$. Note also that the twisting isomorphism $\eta_0 : TQ \cong \Lambda^{2,+}_\nabla$ and the induced metric on $Q$ allows us to identify:

\[ C^\infty(Q,\mathcal{E}_c) \subset \Omega^1(Q,\mathfrak{H}_c) \]

\beq C^\infty(Q,\mathcal{E}_c/W) \subset \Omega^1(Q,\mathfrak{H}_c/W) \label{tchatchatcha} \eeq
as $1$-forms of norm one in the induced metric. 

\begin{remark}
Let us clarify what we mean by $\Omega^1(Q,\mathfrak{H}_c/W)$. As will be clear from the proof of Lemma \ref{bundleball}, if we consider $\mf{V}$ without a preferred metric and orientation, the conic bundle $\mathfrak{H}_c/W$ can be seen as an associated bundle to a principal $C_{GL(2,\C)_-}(\Gamma)$-bundle via the determinant map

\[ \det : C_{GL(2,\C)_-}(\Gamma) \to \C^* \]
since this map allows us to make $C_{GL(2,\C)_-}(\Gamma)$ act on the cone $\mathfrak{h}_c/W$ by rescaling. Then, using Chevalley's restriction theorem: 

\[ \C[\mathfrak{g}_c]^{G_c} \cong \C[\mathfrak{h}_c]^W \] 
any choice of basis of $W$-invariant homogeneous polynomials $\left\{ p_1,\ldots,p_r \right\}$ on $\mathfrak{h}_c$ identifies $\mathfrak{h}_c/W$ with a $r$-dimensional complex vector space $D$ endowed with a $\C^*$-action given by:

\[ \lambda(v_1,\ldots,v_r) = (\lambda^{\deg(p_1)}v_1,\ldots,\lambda^{\deg(p_r)}v_r) \; \; \; , \; \; \; \forall \lambda \in \C^* \] 

Hence we can identify $\mathfrak{H}_c/W$ with a $\C^*$-vector bundle $\mathcal{D} \to Q$ with fibers $D$. As such, when we write $\Omega^1(Q,\mathfrak{H}_c/W)$ we really mean $\Omega^1(Q,\mathcal{D})$. Nevertheless, we choose to keep our notation as in \ref{tchatchatcha} since it emphasizes the geometric origin of $\mathcal{D}$, which will be important in applications.
\end{remark}

In our application to $G_2$-geometry, there will be a natural $G$-structure $\mathcal{P}_G$ on a principal $G_c$-bundle $\mathcal{P}_{G_c} \to Q$, inducing real versions of the objects above. So we consider a bundle of real Cartan subalgebras $\mathfrak{H} := \mathfrak{H}_c \cap Ad(\mathcal{P}_G)$ and also the underlying real bundles $\mathcal{E} := \mathcal{E}_c^\R$ and $\mathcal{E}_W := \mathcal{E}_c^\R/W$. We will prove that closed $\mathfrak{H}/W$-valued $1$-forms define $7$-manifolds with closed $G_2$-structures and a compatible structure of coassociative fibration by hyperk\"ahler ALE spaces over $Q$. If the form is also co-closed, then the $G_2$-structure will satisfy a certain harmonicity condition (see Theorem \ref{teoremasso}). Such manifolds can be realized geometrically inside the total space of $\mathcal{U}$. 

A section $s$ of $\mathcal{E}_W$ defines an object called a \emph{cameral cover} of $Q$, which is a Galois $|W|$-to-$1$ branched covering map:

\[ c_s : \widetilde{Q}_s \to Q \]

This is defined as follows: there is a projection map $e : \mathcal{E} \to \mathcal{E}_W$ which is a $|W|$-to-$1$ cover with Galois group $W$. The cameral cover is defined as the fiber product:

\beq \widetilde{Q}_\theta := \mathcal{E} \times_{\mathcal{E}_W,e,s} Q \eeq
and the covering map is the natural map to $Q$. Crucially, $\widetilde{Q}_s$ comes with an embedding $\widetilde{Q}_s \embedd \mathcal{E}$ given by the pullback $c_s^*s$. 

An important consequence of this discussion is that, in the absence of ramification: 

\beq \Omega^1(Q,\mathfrak{H}/W) \cong \Omega^1(\widetilde{Q}_s,\mathfrak{H})^W \label{important} \eeq

This isomorphism will be used repeatedly below. Note also that given $\widetilde{Q}_s$ and the two maps, one can recover the section $s$. For this reason, we refer to the section $s$ as a {\emph{cameral section}.  We will explain in \ref{higgsstructure} how a commuting Higgs field $\theta$ defines a cameral section $s_\theta$.

Another useful notion is that of a \emph{spectral cover}. These arise once one chooses a linear representation of $G$. Then a section $s$ of $\mathcal{E}/W$ can be seen as a matrix-valued one-form whose characteristic polynomial $f_s$, seen as a section of the tautological cotangent bundle $p_Q: T^*Q \to \text{tot}(T^*Q)$, defines a $|\Gamma|$-to-$1$ branched covering map:

\[ p_Q|_{f_\theta} : Z(f_s) = \left\{ x \in T^*Q ; f_s(x)=0 \right\} \to Q \]

This is called the spectral cover associated to the chosen representation. It can be seen as a Lagrangian submanifold of $T^*Q$ parametrizing the eigenvalues of $s$.

\subsection{Cameral Structures and Harmonic Higgs Bundles} \label{higgsstructure}

The goal of this section is to show that a commuting Higgs field $\theta$ defines a cameral section $s_\theta$.

Let $Q$ be a closed oriented Riemannian $3$-manifold and $\pi_c : \mathcal{P}_{G_c} \to Q$ a (right) principal $G_c$-bundle, with complex structure $J$ and endowed with a $G$-structure, i.e., a reduction to a principal $G$-bundle $\pi: \mathcal{P}_G \to Q$. Choosing a base point $p_0 \in \mathcal{P}_{G_c}$, this can be defined as:

\[ \mathcal{P}_G = \left\{ p \in \mathcal{P}_{G_c} \; | \; \exists g \in G \text{ s.t. } p=p_0g \right\} \]
and $\pi = \pi_c|_{\mathcal{P}_G}$.

\begin{definition} Let $\mathcal{T} \leq G$ be a maximal torus and $N := N_G(\mathcal{T})$ its normalizer. A (unramified) \emph{cameral structure} on $\pi: \mathcal{P}_G \to Q$ is a $G$-equivariant map:

\beq \sigma : \mathcal{P}_G \to G/N \eeq
where $G$ acts on $G/N$ on the left.
\end{definition}

Equivalently, a cameral structure is a section of the bundle $\mathcal{P}_G \times_G (G/N) \to Q$ associated to $\mathcal{P}_G$ via the left $G$-action on $G/N$.

Recall that the (standard) complete flag variety of $\C^n$ is given by $GL(n,\C)/B$ where $B$ is the Borel subgroup of lower triangular matrices. It is isomorphic to $SU(n)/T_n$, where $T_n \cong U(1)^{n-1}$ is the maximal torus of $SU(n)$. A point in this space is determined by a choice of maximal torus in $GL(n,\C)$ and a Borel subgroup containing it, or equivalently, a choice of Cartan subalgebra in $\mathfrak{gl}(n,\C)$ and a Borel subalgebra containing it.

Similarly, $G/\mathcal{T}$ parametrizes pairs of the form $(\mathfrak{h}_c,\mathfrak{b}_c)$ with $\mathfrak{b}_c \subset \mathfrak{g}_c$ a Borel subalgebra containing $\mathfrak{h}_c$. To remove the choice of $\mathfrak{b}_c$ one replaces $\mathcal{T}$ by its normalizer. Thus, the partial flag variety $G/N \subset Gr(r,\mathfrak{g})$ is the parameter space of Cartan subalgebras $\mathfrak{h}_c \subset \mathfrak{g}_c$ (or equivalently, real Cartan subalgebras $\mathfrak{h} \subset \mathfrak{g}$). Of course, it is also the parameter space of maximal tori in $G$.

As a consequence, a cameral structure on $\mathcal{P}_G$ can be equivalently prescribed as a $W$-bundle of maximal tori:

\beq \mathbb{T} \to Q \eeq
or as a $W$-bundle of Cartan subalgebras:

\beq \mathfrak{H} \to Q \eeq

This last bundle can be realized inside of $Ad(\mathcal{P}_G)$ as follows: consider the smooth map $\mathfrak{g}^{reg} \to Gr(r,\mathfrak{g})$ sending a regular (hence semi-simple) element $X$ to the Cartan subalgebra $\mathfrak{h}_X = \mathfrak{c}_\mathfrak{g}(X)$. It has a bundle version: 

\[ \kappa : Ad(\mathcal{P}_G)^{reg} \to Gr(r,Ad(\mathcal{P}_G)) \]

Then $\mathfrak{H} = \overline{\kappa^{-1}(\sigma(Q))}$. Note, however, that in general $Ad(\mathcal{P}_G)^{reg}$ has no global section, hence \emph{there is no global identification of the fibers of $\mathfrak{H}$ to a fixed Cartan $\mathfrak{h}$}.

Similarly, one can realize $\mathbb{T}$ as a sub-bundle of the \emph{conjugate bundle}: 

\[ \mathscr{C} = \mathcal{P}_G\times_G G \] 

One way to see this is by noting that every section of $\mathbb{T}$ defines an automorphism of $\mathcal{P}_G$.

Note also that a cameral structure defines a principal $N$-bundle $\mathcal{P}_N \subset \mathcal{P}_G$, so there is also a conjugate bundle $\mathscr{N} = \mathcal{P}_N\times_N N$.






We turn now to Higgs bundles. Recall that $G$ is a maximal compact Lie subgroup of $G_c$, so complex conjugation on $\mathfrak{g}_c$ is a Cartan involution inducing a Cartan decomposition $\mathfrak{g}_c = \mathfrak{g}\oplus i\mathfrak{g}$. As a consequence, there is a hermitian metric $k$ on $\mathcal{P}_{G_c}$ such that:

\beq Ad(\mathcal{P}_{G_c}) = Ad(\mathcal{P}_G) \oplus_k Ad(\mathcal{P}_G)^\perp \eeq


\begin{definition}
Let $\mathcal{P}_{G_c} \to Q$ be a principal $G_c$-bundle and $G \leq G_c$ a maximal compact subgroup. A \emph{principal $G$-Higgs bundle} is a pair $(\mathcal{P}_G,\theta)$ consisting of a $G$-structure $\mathcal{P}_G$ on $\mathcal{P}_{G_c}$ and $\theta \in \Omega^1(Q,Ad(\mathcal{P}_G)^\perp)$.
\end{definition}

Note that the complex structure $J$ on $Ad(\mathcal{P}_{G_c})$ maps $Ad(\mathcal{P}_G)^\perp$ to $Ad(\mathcal{P}_G)$, so by applying $J$ we can also think of $\theta$ as a one-form taking values in $Ad(\mathcal{P}_G)$. Thus, from now on we consider $\theta \in \Omega^1(Q,Ad(\mathcal{P}_G))$.

Assume that the condition $[\theta\wedge\theta]=0$ holds, and fix an open covering $\mathfrak{U}$ of $Q$ trivializing both $T^*Q$ and $Ad(\mathcal{P}_G)$. Let 

\[\theta^U = (\theta_1^U,\theta_2^U,\theta_3^U) \in \mathcal{C}^\infty(U,\mathfrak{g})^3 \] 
be the components of $\theta$ over $U \in \mathfrak{U}$. Since $\mathfrak{g}$ is compact and simple, for each $q \in U$ and $i =1,2,3$, $\theta_i^U(q)$ is diagonalizable. Since $[\theta_i^U,\theta_j^U]=0$, the three components of $\theta^U(q)$ can be simultaneously conjugated to any given Cartan subalgebra $\mathfrak{h} \subset Ad(\mathcal{P}_G)_q$.

If we are given a cameral structure on $\mathcal{P}_G$ (in the form of a sub-bundle of Cartan subalgebras $\mathfrak{H} \subset Ad(\mathcal{P}_G)$), we can simultaneously conjugate the components of $\theta^U$ to $\mathfrak{H}$ using a local section $g^U$of $\mathscr{C}$. However, there is an ambiguity in the choice of $g^U$ coming from automorphisms of $\mathfrak{H}$: a local section $n^U$ of $\mathscr{N}$ acts on $\mathscr{C}$ by conjugation, and on $\mathfrak{H}$ via the Weyl group. Hence $n^Ug^U$ also conjugates $\theta^U$ to $\mathfrak{H}$, and the answer is different if and only if $n^U$ defines a non-trivial local section of the flat bundle $\mathscr{N}/\mathbb{T} \to Q$, i.e., a non-trivial element $w \in W$.

Performing this construction over all open sets in $\mathfrak{U}$, the upshot is then that the condition $[\theta,\theta]=0$ gives us a cameral section $s_\theta \in \mathcal{C}^\infty(Q,\mathcal{E}_W)$ and a cameral cover $c_{\theta} : \widetilde{Q}_{\theta} \to Q$. We sum up this discussion in the following:

\begin{lemma}
Let $\mathcal{P}_G \subset \mathcal{P}_{G_c}$ be endowed with a cameral structure defined by a bundle of Cartan subalgebras $\mathfrak{H} \to Q$, and let $\mathcal{E} = T^*Q\otimes \mathfrak{H}$. 

Then, any Higgs field $\theta \in \Omega^1(Q,Ad(\mathcal{P}_{G_c})^\perp)$, defines a cameral section:

\[ s_\theta \in \mathcal{C}^\infty(Q,\mathcal{E}/W) \]
and hence also a  $W$-cameral cover:

\[ c_{\theta} : \widetilde{Q}_{\theta} \to Q \] \label{lemata}
\end{lemma}

Finally, we note that a choice of representation $\rho: G \to GL(V)$ will lead to an associated $G$-bundle and a spectral cover $c_{\theta,\rho} : \widetilde{Q}_{\theta,\rho} \to Q$, which is always a finite quotient of $\widetilde{Q}_{\theta}$. 




\subsection{The Main Theorem} \label{mainsection}







The remainder of section \ref{deformation} is devoted to proving our main result:

\begin{theorem} 
\begin{enumerate} Let $(M_0 \to Q,\varphi_0)$ be a coassociative ADE fibration, $\mathcal{P}_{G_c}$ the trivial principal $G_c$-bundle and $\mathcal{P}_G \subset \mathcal{P}_{G_c}$ a $G$-structure.

\item There is a sub-bundle $t : \mathcal{E} \to Q$ of $T^*Q\otimes Ad(\mathcal{P}_{G_c})$ endowed with a $W$-action, and a family $u: \mathcal{U} \to \mathcal{E}/W$ of hyperk\"ahler ALE spaces with the following properties: 

	\begin{itemize}
		\item $\mathcal{U}|_{\ul{\mf{0}}(Q)} \cong M_0$ 
		\item $\mathcal{U}|_{t^{-1}(q)} \cong {\mathcal{R}}$
		\item $u$ is a smooth fiber bundle map on a dense open $\mathcal{U}^\circ \subset \mathcal{U}$
		
	\end{itemize}
	where $\ul{\mf{0}} : Q \to \mathcal{E}$ denotes the zero-section.
	
\item Given $\theta \in \Omega^1(Q,Ad(\mathcal{P}_G)^\perp)$ satisfying:
 \[ [\theta\wedge\theta]=0 \]
 let $s_\theta \in C^\infty(Q,\mathcal{E}/W)$ be the corresponding cameral section, and 
 \[ M_\theta := u^{-1}(s_\theta(Q)) \stackrel{t \circ u}{\to} Q \]
 
Assume there is a flat $G$-connection $\nabla$ on $\mathcal{P}_{G_c}$ such that:

\[ d_\nabla\theta=0 \]

Then, restricting to the smooth locus $M_\theta^\circ = M_\theta\cap \mathcal{U}^\circ$, one can construct a connection $\mf{H}_\theta$ on $M_\theta^\circ \to Q^\circ$ and:
	
	\beq	\twoparteq{\eta_\theta \in \Omega^2_{\mf{H_\theta}}(M_\theta^\circ/Q^\circ) \otimes \Omega^1(Q^\circ)}{\mu_\theta \in \Omega^3(Q^\circ)} \eeq
such that $(\eta_\theta,\mu_\theta, \mf{H}_\theta)$ satisfy Donaldson's criteria, and hence define a closed $G_2$-structure $\varphi_\theta := \eta_\theta+\mu_\theta$ on $M_\theta^\circ$ such that $M_\theta^\circ \to Q^\circ$ is a coassociative ALE fibration. In particular, $\varphi_0 = \eta_{\ul{\mf{0}}} + \mu_{\ul{\mf{0}}}$.



\item If there is a metric $k$ on $\mathcal{P}_{G_c}$ such that:

\[ d_\nabla^{\dagger_k}\theta=0 \]
then $[\eta_\theta] = d h_\theta$ for a harmonic affine section $h_\theta : Q \to H^2(M_\theta/Q,\R)$. In other words, $\varphi_\theta$ satisfies the adiabatic limit of the torsion-free condition in the sense of \cite{donaldson}.
	
	\end{enumerate}
	\label{teoremito}
\end{theorem}

\begin{corollary} \label{corolarao}
	The coassociative ADE fibration $(M_0,\varphi_0) \to Q$ admits a \emph{local moduli space of smoothings} given by ``closed'' cameral sections:
	
	\beq \moduli_{G_2}(M_0) := \Omega^1_{cl}(Q,\mathfrak{H}/W) \eeq
	parametrizing nearby closed $G_2$-structures on coassociative ALE fibrations $M_\theta \to Q$.
	
	Moreover, the subspace
	
	\[ \Omega^1_{k}(Q,\mathfrak{H}/W) \subseteq \moduli_{G_2}(M_0) \]
	of ``$k$-harmonic'' cameral sections parametrizes those $\varphi_\theta$ which also solve the adiabatic limit of the torsion-free condition.
\end{corollary}

The key observation that makes Theorem \ref{teoremito} plausible is the fact that, locally, a cameral section takes values in $\R^3 \otimes \mathfrak{h}/W$, the parameter space of smoothings of the fibers $\R^4/\Gamma$ of $M_0$. We will see that the cocycle defining $\mf{V}$ gives us a cameral structure $\mathfrak{H} \subset Ad(\mathcal{P}_G)$, and hence also the bundle $\mathcal{E} = T^*Q\otimes\mathfrak{H}$. Then, by Lemma \ref{lemata}, any Higgs field $\theta$ gives a cameral section $s_\theta$, which can be thought as a choice of fiberwise smoothings for $M_0$.

The main ingredient to prove Theorem \ref{teoremito} will be, once $\mathcal{E}$ is constructed, to pullback the family $\Upsilon: \mathcal{R} \to \C^2\times_{\mathfrak{h}_c/W} \mathfrak{h}_c$ to $\mathcal{E}$. However, it is apparent that a global pullback is impossible since there is no map $\mathcal{E} \to \C^2\times_{\mathfrak{h}_c/W} \mathfrak{h}_c$. There are two ways to circumvent this issue: one is to work over the universal cover of $Q$, where $\mathcal{E}$ becomes trivial and one can take a global pullback. The drawback then is that one needs to solve for $\pi_1(Q)$-invariant Donaldson data. 

The second way, which is the route we take, is to work over a trivialization of $\mathcal{E}$ and construct $\mathcal{U}$ by gluing together local pullbacks of $\Theta$ using the cocycle of $\mf{V}$.





The proof of Theorem \ref{teoremito} will be carried out in the next three sections.

\subsection{A Fibration of Hyperk\"ahler Deformations} \label{hkdeformations}

Recall from the definition of $\mf{V} \to Q$ in section \ref{adeg2orbi} that $\mathcal{P}_+ \to Q$ is the principal spin bundle (where $SU(2)_+ = Spin(3)$) and that $\mathcal{P}_- \to Q$ is a second principal bundle with structure group $C_{SU(2)_-}(\Gamma)$. More precisely, $\Gamma$-invariance of $\mf{V}$ induces a reduction of the structure group of $\mathcal{P}_-$ to $C_{GL(2,\C)_-}(\Gamma)$, and the metric connection on $\mf{V}$ induces a further reduction to $C_{SU(2)_-}(\Gamma)$. Since we are now proceeding to deform $\mf{V}$ to produce ALE fibrations, the connection on $\mf{V}$ is treated as a variable, so in this section the structure group of $\mathcal{P}_-$ is $C_{GL(2,\C)_-}(\Gamma)$.

\begin{lemma}{(Theorem $4.9$, Part $1$)} \label{bundleball}
	There is a vector bundle $t : \mathcal{E} \to Q$ with $\text{rank}(\mathcal{E}) = 3r$ and a family $u: \mathcal{U} \to \mathcal{E}$ of hyperk\"ahler ALE spaces satisfying:
	
	\begin{enumerate}
		\item $\mathcal{U}|_{\ul{\mf{0}}(Q)} \cong M_0$
		\item $\mathcal{U}|_{t^{-1}(q)} \cong {\mathcal{K}}$
	\end{enumerate}
	where $\ul{\mf{0}} : Q \to \mathcal{E}$ denotes the zero-section.
\end{lemma}

\begin{proof}

Let $Ad(\mathcal{P}_G) \to Q$ be the adjoint bundle. Then at each fiber we have a subspace $\mathfrak{h} \subset \mathfrak{g}$, giving us a sub-bundle of Cartan subalgebras:

\beq \mathfrak{H} \to Q \eeq

This can be constructed explicitly as follows:

First, observe that the group $N_{GL(2,\C)_-}(\Gamma)$ acts on $\C^2$ by automorphisms normalizing $\Gamma$, so it descends to an action by automorphisms on $\C^2/\Gamma$. By universality of $\Psi : \mathcal{S} \to \mathfrak{h}_c/W$, we have that $N_{GL(2,\C)_-}(\Gamma) \leq Aut(\mathcal{S})$. In particular, it acts on the base $\mathfrak{h}_c/W$. Explicitly, the action maps $g \in N_{GL(2,\C)_-}(\Gamma)$ to $\rho_-(g) \in GL(\mathfrak{h}_c)^W$ such that:

\[ \rho_-(g)(X) = (\delta_g(X),\det(g)X) \]
where the first map can be described, using the McKay correspondence, as the map that sends an irrep $\gamma$ of $\Gamma$ to $\gamma \circ C_g$. So the upshot is that in the presence of a principal $N_{GL(2,\C)}(\Gamma)$-bundle, this construction produces a $W$-invariant bundle of Cartans $\mathfrak{H}_c$ and hence also a conic bundle $\mathfrak{H}_c/W$.

In our case of interest, $\mathfrak{H}_c$ is to be constructed from the cocycle of $\mf{V}$. It is clear that $\rho_-(g)$ is just the determinant action for $g \in C_{GL(2,\C)_-}(\Gamma)$. Hence the associated bundle $\mathfrak{H}_c$ is a $\C^*$-bundle, and its intersection with $Ad(\mathcal{P}_G)$ gives the desired bundle $\mathfrak{H} \to Q$ of real Cartans.




We define:

\beq t: \mathcal{E}_c = \sph\Lambda^{2,+}_\nabla \times_{\mathfrak{H}_c/W} \mathfrak{H}_c \to Q \eeq

Recalling that $\Lambda^{2,+}_\nabla$ is the adjoint bundle of $\mathcal{P}_+$, we note that the fibers of $\mathcal{E}$ are identified with $\C^2\times_{\mathfrak{h}_c/W} \mathfrak{h}_c$ - the base of $\Upsilon$. Crucially, the bundle $\mathcal{E}$ is the image of the cocycle $\zeta_\mf{V}$ of $\mf{V}$ via the map

\[ Ad \times \rho_- : G_\mf{V} \to Aut(\C\p^1\times_{\mathfrak{h}_c/W} \mathfrak{h}_c)  \]
where we recall that $G_\mf{V} = SU(2)_+ \times C_{GL(2,\C)_-}(\Gamma)$ is the structure group of $\mf{V}$. More precisely, because $\rho_-$ acts by dilations, all that matters is that $\zeta : \C^2 \to \mathfrak{h}_c/W$ is $(SU(2)_+,\C^*)$-equivariant, which it is since the $SU(2)_+$-action is nothing but a change of complex structure in $\mf{V}$ (this construction is a special case of a more general statement of existence of equivariant fiber products). 

Thus we conclude that the bundle $\mathcal{E}$ is associated to $\mathcal{P}_+\times_Q \mathcal{P}_-$ via the action:

\begin{align*}  \begin{array}{ c c c c c l l}
SU(2)_+\times C_{GL(2,\C)_-}(\Gamma) & \times & \C\p^1\times_{\mathfrak{h}_c/W} \mathfrak{h}_c & \to & \C\p^1\times_{\mathfrak{h}_c/W} \mathfrak{h}_c \\
(g_+,g_-) &  &  ([a],\chi_c) &  \mapsto & (Ad(g_+)[a],\det(g_-)\chi_c)
\end{array}
\end{align*} 
and is naturally a \emph{bundle of HK deformations over $Q$}. 

For $g=(g_+,g_-) \in G_\mf{V}$ and $b = ([a],\chi_c) \in \C\p^1\times_{\mathfrak{h}_c/W} \mathfrak{h}_c$ we denote: 

\[ G_g(b) := (Ad(g_+)[a],\det(g_-)\chi_c) \]

Now let $\mathfrak{U} = \left\{ U_i \subset Q; i \in I \right\}$ be an open cover of $Q$ trivializing both $\mathcal{P}_+$ and $\mathcal{P}_-$ (and in particular, trivializing $t$), and 
\[ \mathfrak{U}_\mathcal{E} := \left\{ U_i \times \C\p^1\times_{\mathfrak{h}_c/W} \mathfrak{h}_c ; U_i \in \mathfrak{U}\right\} \] 
an open cover of $\text{tot}(\mathcal{E}_c)$. We have projection maps:

\[ u_i : U_i \times(\C\p^1\times_{\mathfrak{h}_c/W} \mathfrak{h}_c) \to \C\p^1\times_{\mathfrak{h}_c/W} \mathfrak{h}_c \]

We define:

\beq \mathcal{U}_i := u_i^*(\mathcal{R}) \to U_i\times (\C\p^1\times_{\mathfrak{h}_c/W} \mathfrak{h}_c) \eeq

We would like to patch together the $\mathcal{U}_i$'s, and in order to do so we will define an action of $G_\mf{V}$ on $\mathcal{R}$. 


We know there is an action of $C_{GL(2,\C)_-}(\Gamma)$ by automorphisms on $\mathcal{S}$: 
\[ \rho_{\mathcal{S}} : C_{GL(2,\C)_-}(\Gamma) \times \mathcal{S} \to \mathcal{S} \]
making $\Psi$ equivariant with respect to the determinant action on the base. Since $\mathcal{R}$ is obtained from $\mathcal{S}$ by base change, we can define:

\begin{align*}  \begin{array}{ c c c c c c l l}
\rho_{\mathcal{R}} : & G_\mf{V} & \times &\mathcal{R} & \to & \mathcal{R} \\
& g &  &  (b,p) &  \mapsto & (G_g(b), \rho_{\mathcal{S}}(g_-)p)
\end{array}
\end{align*}

Let $\zeta_\mf{V}$ be the \v{C}ech cocycle of transition functions of $\mf{V}$ associated to the covering $\mathfrak{U}$ and let 

\[ (\rho_{\mathcal{R}})_* : \check{\mathrm{H}}^1(\mathfrak{U},G_\mf{V}) \to \check{\mathrm{H}}^1(\mathfrak{U},Aut(\mathcal{R})) \]
be the map on cocycles induced by $\rho_{\mathcal{R}}$. Now use the cocycle $(\rho_{\mathcal{R}})_*(\zeta_\mf{V})$ to glue the $(U_i,u_i)$'s over $\mathcal{E} := (\mathcal{E}_c)^\R$ into a global family:

\beq u : \mathcal{U} \to \mathcal{E} \eeq



\end{proof}

\subsection{Closed \texorpdfstring{G\textsubscript{2}} --Structures from Closed Cameral Sections} \label{closedg2}

Let $q=t\circ u$. We now have a diagram:

\beq \begin{tikzcd}
\mathcal{U} \arrow[rightarrow]{rd}{u} \arrow[rightarrow]{dd}{q} & \\
& \mathcal{E} \arrow[rightarrow]{ld}{t} \\
Q
\end{tikzcd}
\eeq

Consider the bundle $t: \mathcal{E} \to Q$. It comes with a fiberwise $W$-action which is free outside the zero-section, so there is a $|W|$-to-$1$ map $e: \mathcal{E} \to \mathcal{E}_W := \mathcal{E}/W$, inducing a fiber bundle:

\beq t_W : \mathcal{E}_W \to Q \eeq

Suppose one is given a cameral section $s : Q \to \mathcal{E}_W$ of $t_W$. As mentioned before, this is equivalent to giving a cameral cover $s^*e: s^*\mathcal{E} \to Q$, which we will write simply as $c : \widetilde{Q}_s \to Q$. So we see $s$ as a multisection of $t$ whose image is a Galois $W$-cover of $Q$. The cameral cover will be smooth if $s$ is chosen to be transversal to the discriminant locus $\Delta_\mathcal{E} = \Lambda^{2,+}_\nabla \otimes \Delta_\mathbb{H}$ of $e$, so we will always assume this is the case.

Now, suppose we have $\theta \in \Omega^1(Q,Ad(\mathcal{P}_G))$ such that $[\theta\wedge\theta]=0$. Then $\theta$ defines a cameral section:

\[ s_\theta : Q \to \mathcal{E}_W \] 
and a cameral cover

\[ c : \widetilde{Q}_\theta \to Q \] 

Moreover, $\theta$ determines a $7$-dimensional manifold:

\beq M_\theta := (e\circ u)^{-1}(s_\theta(Q)) \eeq 
which comes with a natural projection map: 
\beq \pi_\theta := q|_{M_\theta} : M_\theta \to Q \eeq

We will phrase most of our constructions below in terms of $\theta$ rather than $s_\theta$. This will cause no confusion as long as the reader remembers that if $\theta \neq \theta'$ are such that $s_\theta = s_{\theta'}$ (i.e., if $\theta'=w(\theta)$ for some $w \in W$), then the corresponding constructions are isomorphic. In other words, the $G_2$-manifolds $(M_\theta,\varphi_\theta)$ we will construct below are parametrized by $s_\theta$ rather than $\theta$.

An important consequence of \ref{important} that will be used repeatedly below is that a the pullback of $\mathcal{E}$ to $\widetilde{Q}_\theta$ has no $W$-action. Hence, locally on $\widetilde{Q}_\theta$, $c^*\theta$ is described by $r$ independent $1$-forms $\theta_1,\ldots \theta_r$.

For the next result, we will need the following construction: recall that $u : \mathcal{U} \to \mathcal{E}$ is a family of hyperk\"ahler ALE $4$-manifolds with a structure of smooth fiber bundle with fibers $S$ in the complement $ \mathcal{E}^\circ$ of a discriminant locus $\Delta_\mathcal{E}$ given by the fibration of root hyperplanes in $\mathfrak{H}_c$ over $Q$. Let $u^\circ : \mathcal{U}^\circ \to \mathcal{E}^\circ$ be the bundle map. Then there is a family of lattices over $\mathcal{E}^\circ$ given by $H^2(\mathcal{U}^\circ/\mathcal{E}^\circ,\Z)$. It gives rise to a flat bundle:

\[ \ell : \mathcal{H} \to \mathcal{E}^\circ \]
with $O(\mathcal{H}) = W$, and a flat \emph{affine} bundle:

\[ \hat{\ell} : \widehat{\mathcal{H}} \to \mathcal{E}^\circ \]
with the same fibers but with structure group an extension of $H^2(S,\R)$ by $W$.

If $M_\theta^\circ$ is now the part of $M_\theta$ contained in $\mathcal{E}^\circ$, we have by restriction the bundles:

\[ \ell_\theta: \mathcal{H}_\theta \to M_\theta^\circ \]
\[ \hat{\ell}_\theta : \widehat{\mathcal{H}}_\theta \to M_\theta^\circ  \]

\begin{lemma}{(Theorem $4.9$, Part $2$)}
Let $\theta \in \Omega^1(Q,Ad(\mathcal{P}_G))$ and let $\nabla$ be a flat $G$-connection on $\mathcal{P}_G$ such that (c.f. \ref{pwequations}):

 \begin{align} 
 \begin{array}{ c c l l l}
	[\theta\wedge\theta] & = & 0 \\
	d_\nabla\theta & = & 0 \\
	\end{array} 
\end{align}                                            

Then there is an element 

\[ \eta_\theta \in \Omega^2(M_\theta^\circ/Q)\otimes \pi_\theta^*\Omega^1(Q) \]
such that 

\beq [\eta_\theta] =dh_\theta \label{duuuh} \eeq 
for some affine section $h_\theta$ of $\widehat{\mathcal{H}}_\theta$.

Moreover, if $\theta=\ul{\mf{0}}$ is the zero section, then $\eta_{\ul{\mf{0}}}= \eta_0$.


	

\end{lemma}

\begin{proof}

	Consider the tautological bundle:
	
	\[ t^*\mathcal{E} \to \text{tot}(\mathcal{E}) \]
endowed with the tautological section $\tau : \text{tot}(\mathcal{E}) \to t^*\mathcal{E}$. We use the McKay correspondence to identify $t^*\mathfrak{H} \cong \mathcal{H}$ over $\mathcal{E}^\circ$. Under this identification, $\tau \circ \theta$ is a $1$-form with values in cohomology classes whose periods are specified by $\theta$.

Let again $\mathfrak{U} = \left\{ U_i ;i \in I \right\}$ be a good open covering of $Q$ trivializing $\mathcal{E}$ and let $\mathfrak{V} = c^{-1}\mathfrak{U}$ be the pullback open covering of $\widetilde{Q}_\theta$.

Over $U \in \mathfrak{U}$, $\nabla$ is equivalent to the trivial connection so $d\theta^U = 0$. Let $V = c^{-1}(U)$. We have $\theta^V = (\theta_1^V,\ldots,\theta_r^V)$ where $\forall i = 1,\ldots,r$, $d\theta_i=0$. This implies that $d(\tau \circ \theta_i)=0$ and thus: 

\[ \tau \circ \theta_i = dh^V_i \] 
for some local function $h^V_i : V \subset \widetilde{Q}_\theta \to \R$. 

Because $\theta$ is a global section, given two open sets $V$, $V'$ and functions $h^V_i$, $h^{V'}_i$, one must have:

\[ dh^V_i|_{V\cap V'} = dh^{V'}_i|_{V\cap V'} \]
and hence:

\[ h^V_i|_{V\cap V'} = h^{V'}_i|_{V\cap V'} + c_{VV'}  \]
where $c_{VV'}$ is a constant.

Passing to the full open covering $\mathfrak{V}$ of $\widetilde{Q}_\theta$ and projecting down to $Q$, the $h_i$'s combine to an \emph{affine} section:

\[ h_\theta : Q \to \widehat{\mathcal{H}}_\theta \] 
such that $\tau \circ \theta = dh_\theta$.

Finally, choosing local frames $dx_i^U$ for $T^*Q$ (with $U \in \mathfrak{U}$), $\tau \circ \theta$ determines for every $x \in Q$ three elements $(\alpha_1^\theta,\alpha_2^\theta,\alpha_3^\theta)$ of $H^2(\pi_\theta^{-1}(Q))$. Since we are only looking over $M_\theta^\circ$, the Torelli theorem implies that there is a local hyperk\"ahler triple $\ul{\omega}_\theta^U$ representing $(\alpha_1^\theta,\alpha_2^\theta,\alpha_3^\theta)$. Since the local form:

\beq \eta_\theta^U = \sum_{i=1}^3 (\omega_\theta^U)_i dx_i^U \eeq
comes from $\tau \circ \theta$, it glues to a global form

\beq \eta_\theta \in \Omega^2(M_\theta^\circ/Q)\otimes \pi_\theta^*\Omega^1(Q) \eeq
which by definition satisfies $[\eta_\theta] =dh_\theta$. The statement for $\theta=\ul{\mf{0}}$ is clear from the construction.

\end{proof}

\begin{corollary} \label{lemamaiscabul}
Let $\theta$ and $\eta_\theta$ be as in the previous lemma. Then there is a connection $\mf{H}_\theta$ on $\pi_\theta$ and a positive $\mu_\theta \in \pi_\theta^*\Omega^3(Q)$ such that $(\eta_\theta,\mu_\theta,\mf{H}_\theta)$ satisfies Donaldson's criteria: 
	
	\beq \twoparteq{d_f\eta_\theta = 0}{d_{\mf{H}_\theta}\eta_\theta=0} \qquad \qquad \twoparteq{d_f\mu_\theta=-F_{\mf{H}_\theta}(\eta_\theta)}{d_{\mf{H}_\theta}\mu_\theta=0}  \label{donaldissimo}\eeq 
	
	It follows that
\beq \varphi_\theta := \eta_\theta + \mu_\theta \eeq
is a closed $G_2$-structure on $M_\theta^\circ$ making $\pi_\theta : M_\theta \to Q$ a (generically) coassociative ALE fibration.
	
\end{corollary}

\begin{proof}

The first statement follows from the previous lemma and Proposition \ref{donaldao2}. The existence of $\mu_\theta$ also follows from \ref{donaldao2}, since $\nabla_0$ and $\eta_0$ satisfy the decay conditions by hypothesis, and hence so does $\eta_\theta$.

\end{proof}

As a consequence of these results:

\begin{corollary} The Higgs field $\theta$  defines a $7$-manifold $M_\theta$ with a closed $G_2$-structure $\varphi_\theta$ and a compatible fibration $M_\theta \to Q^\circ$ with coassociative ALE fibers, where $Q^\circ = \theta^{-1}(\mathcal{E}^\circ)$ is dense and open in $Q$. The pair $(M_{\theta},\varphi_\theta)$ is a deformation of $(M_0,\varphi_0)$ as coassociative fibrations.
\end{corollary}

We can now provide a good visualization of our family of $7$-manifolds. Consider the diagram:

\beq \begin{tikzcd}
\mathcal{F}  \arrow[rightarrow]{d}{w:=\tau^*(e\circ u)} \arrow[bend right=90,swap]{dd}{f} & \mathcal{U} \arrow[swap]{d}{e\circ u} \ar[bend left=30, near end]{dd}{q}  \\
Q \times C^\infty_{cl}(Q,\mathcal{E}_W) \arrow[rightarrow]{r}{\tau} \arrow[swap]{d}{\pi_2}  & \mathcal{E}_W  \arrow[swap]{d}{t_W}  \\
C^\infty_{cl}(Q,\mathcal{E}_W) & Q 
\end{tikzcd}
\eeq

Here, $\tau$ is the \emph{evaluation map}: $\tau(q,s) := s(q)$ and $\mathcal{F}$ is the pullback of $\mathcal{U}$ by $\tau$. From now on, we write $\mathcal{B} :=  C^\infty_{cl}(Q,\mathcal{E}_W)$. 

Our family of interest is $f : \mathcal{F} \to \mathcal{B}$. For every section $s \in \mathcal{B}$, $M_s = f^{-1}(s)$ is a $7$-manifold with a closed $G_2$-structure $\varphi_s$ and a generically coassociative ALE fibration $\pi_s := w|_{M_s} : M_s \to Q$ deforming $f^{-1}(\mf{0}) = (M_0,\varphi_0)$. Hence one should think of $\mathcal{B}$ as the local moduli space of deformations of $(M_0,\varphi_0)$ as a coassociative fibration over $Q$.

\subsection{Harmonic Metrics and the Adiabatic Limit} \label{aftersection}

So far, the geometric picture we have described is independent of the harmonic metric equation (also known as the D-term):

\beq d_\nabla^{\dagger_g} \theta = 0 \label{ugabuga} \eeq

Indeed, the deformations $\varphi_\theta$ of the closed $G_2$-structure $\varphi_0$ on $M_0$ constructed in the previous sections are parametrized by spectral/cameral covers $s_\theta$ associated to commuting solutions $(\nabla,\theta)$ to the F-term equations:

\[ F_\nabla = [\theta\wedge\theta] = 0 \]
\beq d_\nabla\theta=0 \label{opamermao}\eeq

We will show that equation \ref{ugabuga} is related to torsion-free deformations, i.e., those satisfying $d*_{\varphi_\theta} \varphi_\theta=0$ More precisely, we will show the following:

\begin{theorem}{(Theorem $4.9$, Part $3$)} Let $(M_0,\varphi_0)$ a coassociative ADE fibration and $(\nabla, \theta)$ a solution of \ref{opamermao} taking values in the McKay bundle $\mathcal{P}_G$. Let $c : \widetilde{Q}_\theta \to Q$ be its cameral cover and $(M_\theta,\varphi_\theta)$ the corresponding deformation of $(M_0,\varphi_0)$. 

Then a solution to the harmonic metric condition $d_\nabla^{\dagger_g} \theta = 0$ is equivalent to a solution to the \emph{adiabatic limit of the torsion-free condition} \cite{donaldson}: 

\beq \text{MC}(h_\theta)=0 \label{maxim}\eeq
where $h_\theta : Q \to \widehat{\mathcal{H}}_\theta$ is an affine section such that $\tau\circ \theta=dh_\theta$, and $\text{MC}$ stands for mean curvature. \label{teoremasso}
\end{theorem}

\begin{remark} In the setup of \cite{donaldson}, the section $h_\theta$ takes values in $H^2(K3,\R)$, where the intersection product is indefinite. In that case, $h_\theta$ must also be positive, a condition that is void in the ALE setup. Another distinction is that, for the same reason, a solution to \ref{maxim} will be a minimal, rather than maximal, submanifold. 
\end{remark}







\begin{proof}




Suppose $g$ is a metric on $Ad(\mathcal{P}_G)$ such that $d_\nabla^{\dagger_g}\theta=0$. Then $c^*\theta \in \Omega^1(\widetilde{Q}_\theta,\mathfrak{H})$ so on an open set $V \in \mathfrak{V}$ we can write:

\[  \theta^V := c^*\theta|_V = (\theta_1^V,\ldots,\theta_r^V) \]
where each $1$-form $\theta_i^V \in \Omega^1(V)$ satisfies:

\beq d^{\dagger_g}\theta_i^V=0 \label{ubah}\eeq
with respect to the induced metric (also denoted $g$) on $\mathfrak{H} \subset Ad(\mathcal{P}_G)$.

Since we also have $d\theta_i^V=0$, we can write

\[ \theta_i^V = dh_i^V \]
for a function $h_i^V : V \to \R$. It follows from \ref{ubah} that $h_i^V$ is harmonic, so

\[ h^V = (h_1^V,\ldots,h_r^V) \in C^\infty(V,\mathfrak{H}) \]
is a $g$-harmonic function. Recalling that the McKay correspondence identifies $c_s^*\mathfrak{H} \cong \widehat{\mathcal{H}}_\theta$, projecting to $U = c(V)$ we get a $g$-harmonic affine function:

\[ h^U : U \to \widehat{\mathcal{H}}_\theta \]

Since $h_i^V = \pi_i|_{Im(h^V)}$ are restrictions of linear maps, $Im(h^V)$ is a minimal submanifold with respect to $g$, and so is $Im(h^U)$.

The converse is obtained by running the argument backwards.


\end{proof}

\section{Duality and Character Varieties} \label{duality}

The goal of this section is to prove the duality explicitly for the Hantzsche-Wendt $G_2$-orbifold $N_0 := \C^2/\Z_2\times_\K \G_6$ (c.f. Example \ref{exemplodobob}) and also verify that Corollary \ref{corolarao} computes the correct moduli space of $G_2$-structures for this example. Before we proceed to the calculation, we briefly explain the content of M-theory/IIA duality for this specific geometric context. 

\subsection{M-theory/IIA duality}

Assume $M$ is a $G_2$-orbifold containing a three-manifold of $A_n$ singularities\footnote{The description here only applies to singularities of type A. The type D case is slightly subtler \cite{sen} and the E case is unknown to the author.}, and that $\widetilde{M}$ is a desingularization of $M$ also with $G_2$-holonomy. M-theory compactified on $\widetilde{M}$ contains certain objects called ``$M2$-branes'' that can wrap the exceptional two-cycles of $\widetilde{M}$. The mass of a $M2$-brane is proportional to its area, hence in the limit where we blow down all cycles we get $n$ massless $M2$-branes located at the singularities of $M$.

Now assume $X$ is a Calabi-Yau threefold. In Type IIA superstrings compactified on $X$ there are objects called ``$D6$-branes'' that can wrap calibrated $3$-cycles, i.e. special Lagrangian submanifolds. These cycles come equipped with complex line bundles with a $U(1)$-connection on it. Given a configuration of $n$ $D6$-branes, one can have strings stretched between any two of them. Now the massless limit is obtained by smashing all $D6$-branes together, since the mass of a string is proportional to its length. This limit is described by a $SU(n)$-connection on a vector bundle over a special Langrangian submanifold.

To explain how $X$ is related to $M$, we note that the duality we aim to describe still holds after we take a ``weakly-coupled limit'' of the two theories. The limits are supersymmetric theories: 11-dimensional supergravity for M-theory and 10-dimensional supergravity (with branes) for type IIA. These two theories are related by dimensional reduction on a circle. Moreover, the supersymmetry condition determines the geometric structures these theories classify: they are determined by dimensionally reducing the equation for a 11 or 10 dimensional parallel spinor down to $7$ or $6$ dimensions, respectively. In the first case, one obtains stationary points of a $7$-dimensional analogue of the Chern-Simons functional $CS(\varphi_\C) := \int_M \varphi_\C \wedge d\varphi_\C$, which are exactly the integrable complexified $G_2$-structures. In the second case, one obtains the Hermitian-Yang-Mills equations.

This discussion motivates the following: suppose $M$ is a $G_2$-orbifold of type $A_n$ endowed with a $U(1)$-action by isometries with fixed set $Q \subset M$. The Calabi-Yau space $X$ is called a \emph{IIA dual} for $M$ if it satisfies the following \cite{ach} \cite{atiyahwitten}:
													
													\begin{enumerate}
													  \item $X := M/U(1)$ as topological spaces.
													  
													  \item The complex structure $J$ on $X$ has a real structure such that $Q/U(1) \cong Q$ is a totally real special Lagrangian submanifold.
													  
													  \item There are $n$ $D6$-branes ``wrapping'' $Q \subset X$.
													  
													  \end{enumerate}

													The IIA moduli space parametrizes the following objects:
													
													\begin{enumerate}
													\item Complex structures on $X$ in which $Q$ is totally real.
													\item Complexified K\"ahler structures on $X$
													\item A supersymmetric configuration of $n$ $D6$-branes wrapped on $Q$.
													\end{enumerate}
													 
Up to now the discussion has been general; now we focus on our main example. Let $N_0 = \C^2/\Z_2 \times_\K \G_6$ be the Hantzsche-Wendt $G_2$-orbifold. Its IIA dual is $X = T^*Q$ endowed with a rank $n$ vector bundle $E \to X$ and a Hermitian Yang-Mills connection on $E$. With the metric on $Q$ fixed, there is a unique complex structure on $T^*Q$ under which $Q$ is totally real. This is the complex structure that makes the semi-flat metric on $T^*Q$ a Calabi-Yau metric.

The Hermitian-Yang-Mills condition describes generic configurations of $n$ $D6$-branes. The massless limit where they wrap the special Lagrangian $Q$ is obtained by dimensional reduction of the HYM equations along the leaves of $T^*Q \to Q$. A quick computations shows that the result consists exactly of equations \ref{pwequations}.

\subsection{The Duality for the Hantzsche-Wendt \texorpdfstring{G\textsubscript{2}} --Orbifold} \label{exemploexemplar}
				          													
According to the previous discussion, the IIA moduli space is computed by solutions to \ref{pwequations}, i.e., by the character variety $\text{Char}(\pi,SL_2(\C))$ where $\pi = \pi_1(\G_6)$ is the Hantzsche-Wendt group. Recall that we have the exact sequence $1 \to \Z^3 \to \pi \to \K \to 1$, with $\K \cong \Z_2\times \Z_2$. From section \ref{charvarbieber}, we know that the character variety of the three-torus $\T$ is given by:
											          													
											          													\beq \text{Char}(\Z^3,SL(n,\C)) \cong (\C^*)^{3n-3}/\Sigma_n = \prod_{i=1}^3\big( (\C^*)^{n-1}/\Sigma_n \big) \label{galileoo} \eeq
											          													where $\Sigma_n$ acts by permutations on $(\C^*)^{n-1} \cong \left\{ z_1z_2\ldots z_n=1 \right\} \subset (\C^*)^n$.
											          													
											          													In section \ref{charvarbieber}, we determined that there is a map $r : \text{Char}(\pi,SL(n,\C)) \to \text{Char}(\Z^3,SL(n,\C))$. Moreover, there is a $\K$-action on this last space, given in terms of the presentation \ref{galileoo} by 
											          													
											          													\beq (i,j)[(z_1,z_2,z_3)] = [(z_1^i,z_2^j,z_3^{ij})] \label{ghty} \eeq
											          													where $i,j \in \left\{ \pm 1 \right\}$.
											          													
											          													The main result in section \ref{charvarbieber} was that $Im(r)$ is contained in $\text{Fix}(\K)$ (and equal under an additional condition). The image determines $\text{Char}(\pi,SL(n,\C))$ possibly up to a finite cover given by non-trivial representations of $\K$ mapping to the same element of $\Hom(\pi,SL(n,\C))$.
											          													
											          													When $n=2$ it is easy to describe $\text{Fix}(\K)$. Let $\alpha = (1,-1)$. Then:
											          													
											          													\[ \emph{Fix}(\alpha) = [(\pm 1, \C^*, \C^*)] \cup [(\C^*,\pm 1,\pm 1)] = (\C^2)_{z_2,z_3} \cup \C_{z_1}  \]
											          													Notice that the first factor is contributed by $-1 \in \Z_2$ and the second by $1 \in \Z_2$.
											          													
											          													We can play a similar game for the other two non-trivial elements of $\K$. Hence:
											          													
											          													\beq \text{Fix}(\K) = \bigcap_{(i,j,k) \in \langle (1,2,3) \rangle } \big( (\C^2)_{z_i,z_j}\cup\C_{z_k} \big) = \C_{z_1}\cup\C_{z_2}\cup\C_{z_3} \label{buquezao} \eeq
											          													
											          													Thus $\text{Fix}(\K)$ is a bouquet of three complex lines touching at a point, which we schematically denote by $\mf{Y}_\C$. The image $Im(r)$ can be computed directly from a presentation of $\pi$ to show that $r$ is in fact surjective onto $\text{Fix}(\K)$; essentially, this is because an element in the bouquet, say $(a,0,0)$ is a representation which is non-trivial only at a single generator, so will be automatically a representation of $\pi$. This determines one connected component of the character variety:
											          													
											          													
											          															\beq \text{Char}^0(\pi,SL(2,\C)) \cong \mf{Y}_\C := \C_{z_1}\cup\C_{z_2}\cup\C_{z_3} \label{finntroll} \eeq

											   Up to conjugation, there are also 3 other representations of $\K$ that map to $Ker(\overline{r})$. These correspond to rigid representations, and include the trivial representation of $\pi$.
											          													
											          													Now, the duality predicts that this character variety can be computed as the moduli space of complexified $G_2$-structures $\moduli_{G_2}^\C$ on $N_0=\C^2/\Z_2 \times_\K \G_6$. As mentioned in the introduction, the space $\moduli_{G_2}$ for this example was computed by Joyce \cite{joyce}. The smoothings of the singularity $\C^2/\Z_2$ are obtained either via resolution (blow-up) or deformation. For the full $M_0$, one needs to smooth the fibers consistently with the $\K$-action, i.e., the $\K$-action must lift to an action on the smooth fibers that is asymptotic to the original action. With these constraints, there are three families of smoothings: one family of resolutions, given by: 
					
					\beq Y_1 := \widetilde{\C^2/\Z_2} \times_\K \G_6 \eeq						          											
and two families of deformations, given by 
											          													
											          													  	\[  Y_2 := \left\{ (z_1,z_2,z_3,\epsilon) ; z_1^2+z_2^2+z_3^2=\epsilon \right\} \subset \C^3\times \R^+ \]
											          														\beq  Y_3 := \left\{ (z_1,z_2,z_3,\epsilon) ; z_1^2+z_2^2+z_3^2=-\epsilon \right\} \subset \C^3\times \R^+ \eeq

The resolved family is parametrized by the volume of the blown-up $\p^1$, and the two deformation families $Y_2$, $Y_3$ are parametrized by the volumes of the totally real $\sph^2 \subset Y_2$ and the totally imaginary $\sph^2 \subset Y_3$, respectively. The intuition here is that (say, in the resolved case) once we have resolved one fiber $\C^2/\Z_2$ we are free to choose a K\"ahler class in $\mathfrak{u}(1)$ up to scaling, and once that is chosen flatness of the vertical section $\ul{\omega}$ fixes the volume in all other fibers.

Therefore the moduli space of smooth $G_2$-structures is:

\beq \moduli_{G_2} = \mf{Y}_\R := \R^+_{x_1} \cup \R^+_{x_2} \cup \R^+_{x_3} \eeq
i.e., it consists of three copies of $\R^+$ touching at the origin. We recall that the complexified space $\moduli_{G_2}^\C$ is, away from the discriminant locus, a K\"ahler space admitting a Lagrangian torus fibration over $\moduli_{G_2}$. The torus fibers parametrize holonomies of the ``$C$-field'', i.e., elements in $H^3(M,U(1))$. So topologically, $ \moduli_{G_2}^\C = \mf{Y}_\C$. We have proved:

\begin{theorem} \emph{(M-theory/IIA duality for $N_0$)}: The moduli space of complexified $G_2$-structures on $N_0$ is isomorphic to the non-trivial component of the $SL(2,\C)$ character variety of $\G_6$:

\beq \moduli_{G_2}^\C \cong  \text{Char}^0(\pi,SL(2,\C)) \eeq

						\end{theorem}					          													
											          													
Notice that the character variety description by itself does not tell us much about the $G_2$-structures per se. However, from the construction of our deformation family, we must have:											          													
\beq \moduli_{G_2} = \Gamma_{\text{flat}}(\G_6,T^*\G_6\otimes \mathfrak{u}(1)/\Z_2) \eeq
											          															
These parametrize $\Z_2$-equivalence classes of flat sections of $T^*Q$. Flat sections are those that are fixed by the monodromy group $\K$. Recall that the action of $\K=\langle \alpha, \beta \rangle$ on $T^*Q$ is given in local coordinates by: 

\[ \alpha(dx_1,dx_2,dx_3) = (-dx_1,-dx_2,dx_3) \]

\[ \beta(dx_1,dx_2,dx_3) = (-dx_1,dx_2,-dx_3) \]

Because of this, the computation of the fixed set here is completely analogous to the one performed above for the character variety (compare formula \ref{ghty}), the only differences being that the variables are now real and the action is additive rather than multiplicative. Hence, the fixed set is given by $\mf{Y}_\R$, in agreement with the result of Joyce.

In fact, the $G_2$-structures on the smoothings constructed by Joyce are not only closed but also \emph{torsion-free}. Thus they have metrics with $G_2$-holonomy. The techniques developed in this work do not attack the torsion-free condition, but presumably the condition can be reformulated as an analytic condition on the spectral covers. At least in the adiabatic limit proposed by Donaldson \cite{donaldson}, the relevant condition is that the spectral cover must be stationary under mean curvature flow. In this restricted context of flat three-manifolds, the solutions are exactly the flat sections of $\mathcal{E}_W$.  The importance of the adiabatic limit is that torsion-free solutions are supposed to be constructed as formal power series extending an adiabatic solution. So we propose that this could be reformulated in our context via appropriate perturbations of flat spectral covers.
										          													
										          														\begin{remark}
										          															We mentioned before that there are other connected components of $\text{Char}(\pi,SL(2,\C))$ given by three isolated points. Under duality, presumably these three points correspond to rigid $G_2$-structures on $N_0$, i.e. that admit no smoothings - or at least no smoothings preserving the structure of a coassociative fibration.
										          														\end{remark}
											          										
Now, the beauty of the character variety description is that it allows us to generalize the computation of the moduli space to higher rank groups. E.g., suppose now we would like to investigate $G_2$-structures on $\C^2/\Z_n \times_{D_8} \G_6$. Then we would need a generalization of \ref{buquezao} characterizing $Im(r) = \text{Fix}(\K)$. We conjecture the following formula holds for $\text{Fix}(\K)$:

											          													\beq \text{Fix}(\K) \cong \bigcup_{K,K' \leq\K} \bigg( \big(\bigslant{T^3}{K}\big)^{\Sigma_n} \cap \big(\bigslant{T^3}{K'}\big)^{\Sigma_n} \bigg) \label{buruguzinha} \eeq
											          													where $\Sigma_n$ is the permutation group and the union is over distinct proper subgroups of $\K$. 
											          													
											          													For $n=2$ a simple calculation shows that the intersecting factors are perpendicular $\C^2$'s, resulting in an union of three $\C$'s as in formula \ref{buquezao}. Hence, formula \ref{buruguzinha} is correct for $n=2$. Computations by the author suggest that $\text{Char}^0(\pi,SL(3,\C))$ is also given by a trivalent vertex, with components not $\C$ but the reduced scheme associated to a rather complicated scheme. In our forthcoming work \cite{barb} we will address this computation from the perspective of Hilbert scheme of points on the SYZ mirror of $X$.
											          													


										\section{Spectral Correspondence} \label{spectrality}
										
													In this section we prove Theorem \ref{flatspec}, a spectral correspondence for harmonic Higgs bundles $(E, k, \nabla, \theta)$ with structure group $G_c = GL(n,\C)$ satisfying the following conditions: 
													\begin{enumerate}
												\item $[\theta\wedge\theta]=0$
												\item The spectral cover $S_\theta$ in  the fundamental representation of $GL(n,\C)$ is either \emph{unramified} or \emph{totally ramified}.
													\end{enumerate}

												We start with the linear theory. Let $V$ be a complex vector space of dimension $n$ and $\phi \in End(V)$. When $\phi$ is diagonalizable, it can be reconstructed from its eigenvalues
												\[ \ul{\lambda} = \left\{ \lambda_1, \ldots \lambda_n \right\} \] 
												the decomposition of $V$ into $\phi$-eigenlines 
												\[ V = L_1 \oplus \ldots \oplus L_n \] 
												and a surjective function $m: \ul{L} \to \ul{\lambda}$. We refer to $(\underline{\lambda}, \underline{L},m)$ as the spectral data associated to $(V,\phi)$.
													
		 Let $p_i(\phi)$ be the coefficient of $\lambda^{n-i}$ in the expansion of $\det(\lambda\mf{1}-\phi) \in \C[t]$. Consider the map:
													
													\begin{alignat}{2} f : &    End(V)    \parbox{1.6cm}{\rightarrowfill}      \C^n \\
																	  		&    \phantom{\hspace{.5cm}}   \phi      \longmapsto           (p_1(\phi),\ldots,p_n(\phi))
													\end{alignat}
													
													Then it is clear that the eigenvalues $\left\{ \lambda_1, \ldots \lambda_n \right\}$ of $\phi$ depends only on $h(\phi)$. The map $f$ is a prototype of the Hitchin map \ref{hit} defined below.
													
													Now, let $Q$ be a manifold, and $E \to Q$ a rank $n$ complex vector bundle. Suppose $\phi \in \Gamma(Q,End(E))$. Then to each $\phi_q : E_q \to E_q$ we can associate its spectral data $(\underline{\lambda}_q,\underline{L}_q)$. We think of $\phi$ as a ``twisted family'' of endomorphisms parametrized by $Q$.
													
												As $(\underline{\lambda}_q,\underline{L}_q)$ varies over $Q$, it defines: 
													
														\begin{itemize}
															\item a (possibly singular) subspace of $Q\times \C$:
													
																	\begin{align*} \widetilde{Q} & = \left\{ (q,\lambda) ; \lambda \text{ is an eigenvalue of } \phi_q \right\} \\
																												& = \left\{ (q,\lambda) ; \det(\lambda \mf{1}_{E_q} - \phi_q) = 0 \right\} 
																\end{align*}
																	called the \emph{spectral cover} of $Q$ associated to $\phi$. It comes equipped with a generically $n : 1$ covering map $\pi : \widetilde{Q} \to Q$. 
																	
																	\begin{itemize}
																		\item If $\phi$ is \emph{generic} - i.e., diagonalizable with distinct eigenvalues at every point - then $\pi$ is unramified. 
																	
																		\item If $\phi$ has repeated eigenvalues, then its \emph{ramification locus} is given by:
																	
																				\[ \Delta_\pi = \{ q \in Q |  \phi_q \text{ has a multiple eigenvalue} \} \]
																	\end{itemize}
															
														\item A \emph{spectral line bundle}:
															
															\beq m : \mathcal{L} \to \widetilde{Q}  \eeq
																defined as follows: consider the matching maps $m_q : (L_q)_i \mapsto (\lambda_q)_i$. Then $\mathcal{L} = \sqcup_{q \in Q, i} (L_q)_i$ and $m|_{L_q} = m_q$.
														\end{itemize}
														
 Conversely, in nice cases (e.g. if $Q$ is an algebraic variety and $\phi$ is regular, i.e., has one Jordan block per eigenvalue), then given $(\til{Q},\mathcal{L})$, Higgs data can be recovered by $E = \pi_*\mathcal{L}$ and $\phi = \pi_*\tau$, where the tautological section $\tau : \til{Q} \to End(\mathcal{L})$ is defined as $\tau(q,\lambda) = \lambda\mf{1}_\mathcal{L}$.

													\begin{remark}
													\begin{enumerate}
\item														If $\phi$ is  irregular - i.e., has multiple Jordan blocks per eigenvalue - then the pushforward of the spectral line bundle by $m$ will not recover $E$: one needs to consider a more general sheaf $\mathcal{L}' \to \widetilde{Q}$ such that on the locus $\Delta_\phi \subseteq Q$ where $\phi_q$ is irregular, $\mathcal{L}'_{q,\lambda_q}$ jumps in rank due to the presence of multiple eigenlines for the same eigenvalue. Such a locus is codimension two in $Q$. In particular, when $Q$ is a $3$-manifold, $\Delta_\phi$ is a graph in $Q$. We will have more to say about this in the next section.

\item If one is not working with algebraic spaces (as is our case in this paper), then it is important to be careful with the pushforwards. This is the main reason we will consider only unramified or totally ramified Higgs bundles. Fortunately, this will suffice for the applications in \cite{barb}. However, as we explain in the next section, partial ramification is expected to play a fundamental role in $G_2$-geometry.
\end{enumerate}
													\end{remark}
													
													This is the rawest form of the spectral correspondence. One can also consider more general notions of Higgs bundles: one can ``twist'' $\phi$ by requiring its coefficients to be valued in a sheaf of commutative groups $\mathcal{F}$ over $Q$, and also require $\phi$ to satisfy some constraint (e.g., being compatible with fixed geometric structures on $Q$ or $E$). In this situation, the spectral data must be suplemented with additional structure in order to reconstruct $(E,\phi)$. We will be interested in the case of harmonic Higgs bundles over a $3$-manifold $Q$, i.e., $\mathcal{F}$ is the sheaf $\Omega^1_Q$ and we impose equations \ref{pwequations} together with $[\theta\wedge\theta]=0$.

													\begin{definition} Let $(E,k,\nabla,\theta)$ be a harmonic Higgs bundle over $Q$. The \emph{spectral cover} associated to $\theta$ is the space $S_\theta \subset T^*Q$ defined by:
													
															\beq S_\theta = \left\{ (q,s_q) ; \det( s_q \otimes \mf{1}_{E_q} - \theta_q ) = 0 \right\} \eeq
												\end{definition}
													
													\begin{definition} \emph{Spectral Data - unramified case:} Let $(E,k,A, \theta,\ell)$ be a rank $n$ harmonic Higgs bundle over a compact Riemannian $3$-manifold $Q$ with Levi-Civita connection $\delta$. Assume $\theta$ is generic. We define \emph{spectral data} to be:
																				\begin{enumerate}
																				\item A $n$-sheeted, unramified covering map $\pi: S_\theta \to Q$ given by the characteristic polynomial of $\theta$.
																			\item A line bundle $\mathcal{L} \to S_\theta$ determined by the eigenlines of $\theta_q$
																				\item A hermitian metric $\widetilde{k}$ on $\mathcal{L}$ determined by $k$
																				\item A hermitian flat connection $\til{A}$ on $\mathcal{L}$ determined by $A$
																				\item A Lagrangian embedding $\ell : S_\theta \to T^*Q$
																				\end{enumerate}
													\end{definition}
													
													
											\begin{definition} \emph{Spectral Data - totally ramified case:} With the same notation as before, suppose $\theta$ is central - i.e., diagonalizable with all eigenvalues equal. Then its spectral data is as before, except that $\mathcal{L}$ is replaced by a rank n complex vector bundle $\mathcal{E} \to S_\theta$. Moreover, note that $\pi$ is now totally ramified.
													\end{definition}

												We now come to the main result of this section. Let \textbf{Higgs} be the set of harmonic Higgs bundles $(E,k,A, \theta)$ over a compact Riemannian $Q$ and \textbf{Spec} the set of spectral data $(\pi,\mathcal{L},\widetilde{k},\til{A},\ell)$ on $Q$.

																\begin{proposition} \label{flatspec} \emph{(Spectral correspondence for harmonic Higgs bundles)} There is a bijection:
													
																				\beq \textbf{Higgs} \longleftrightarrow \textbf{Spec} \eeq
																			where harmonic Higgs bundles are taken to be either unramified or totally ramified, and the spectral data is chosen appropriately for each case.
																\end{proposition}
																
																\begin{proof}

																Given a harmonic Higgs bundle $(E,k,A,\theta)$, we already know how to construct $\pi : S_\theta \to Q$, the Larangian embedding $\ell$ and $\mathcal{L} \to S_\theta$ such that $E = \pi_*\mathcal{L}$. Note that $\mathcal{L}$ is a subbundle of $\pi^*E$, so the metric and flat connection can be defined by pullback and restriction: $\til{k} := \pi^*k|_{\mathcal{L}}$ and\footnote{Note that unlike the case of general connections, one can restrict a flat connection to a sub-bundle simply by considering the subsheaf of locally constant sections taking values in that sub-bundle.} $\til{A} := \pi^*A|_{\mathcal{L}}$.

															Now, use the hamonicity condition on $k : \til{Q} \to G/K$ to identify (locally) $\theta=dk$. The condition $\nabla_A\theta = d\theta + A\wedge\theta = 0$ can be written as equations for the components of $\theta$ under the identification. Since $\nabla_A^2=0$, we can locally gauge away $A$, so that that the equations become $d\theta=0$.  Let $(x_i,y_i)$ be coordinates in $T^*Q$ such that the standard symplectic form is $\omega=\sum dy_i\wedge dx_i$ and dualize $\theta=\sum\theta_i(x)dx_i$ via the semi-flat metric on $T^*Q$ to obtain $\theta'=\sum\theta_i(x)dy_i$. The embedding given by $\ell(q,s_q) = \det(s_q\otimes \mf{1}_E - \theta_q)$ is Lagrangian if and only if:
																
																\beq \omega|_{\ell(Q)} = \sum_{i=1}^3 d\theta' \wedge dx_i = \sum_{i=1}^3 d(\theta_idx_i) = d\theta = 0 \eeq

															Conversely, given spectral data $(\pi, \mathcal{L},\til{k},\til{A},\ell)$, one constructs $E$ and $\theta$ as usual, and $k = \pi_*\til{k}$, $A=\pi_*\til{A}$ are well-defined because $\pi$ is a local diffeomorphism. The conditions $\nabla_A^2=0$ and $\nabla_Ah=0$ follow from the same conditions for $(\til{h},\til{A})$. Finally, the condition $\nabla_A\theta=0$ is obtained simply by reversing the above argument for the section $\ell$ to be Lagrangian.
																
																\end{proof}
													

												\begin{definition} \label{hit}
											For a general choice of complex reductive Lie group $G_c$ with $r = \text{rank}(G_c)$, a choice of basis of $Ad$-invariant polynomials $\left\{ p_1,\ldots,p_r \right\}$ of degrees $d_1,\ldots,d_r$ in $\mathfrak{g}_c$ defines a \emph{Hitchin map}:
													
											\begin{align*} \mathscr{H} : \textbf{Higgs} & \to \bigoplus_{i=1}^r \mathcal{C}^\infty(Q,\text{Sym}^{d_i}(\slashed{S}_Q)^\R) \\
																						 (E,h,A,\theta) & \mapsto (p_1(\theta),\ldots,p_r(\theta)) \end{align*}
													\end{definition}
												
												In \cite{barb} we will study the Hitchin map in a specific example related to the mirror $X^\vee$ of $X=T^*\G_6$, where we will show that $\mathscr{H}$ can be used to define a smooth model of the SYZ fibration of $X^\vee$. 
												
												For now, we observe two differences from the holomorphic category: one, $\mathscr{H}$ is not surjective in general; and two, even when restricted to its image, $\mathscr{H}$ will not define an integrable system structure, since character varieties of compact three-manifolds do not come equipped with symplectic structures.

\section{Associatives, Ramified Higgs Fields and Isolated Singularities} \label{remakedaremarca}

Our main point in this work is that the spectral/cameral cover approach provides a conceptual framework to study the deformation theory of coassociative fibrations. In line with this idea, we will now suggest that two seemingly different proposals \cite{danonao}, \cite{joycekarigiannis} for constructing isolated (i.e. codimension-seven) singularities on $G_2$-manifolds both fall under the umbrella of \emph{ramified Higgs fields}. 

We will start by explaining the relation between our construction and the work of Joyce and Karigiannis \cite{joycekarigiannis}. Their paper shows how to construct \emph{integrable} $G_2$-structures on a $G_2$-orbifold $M \to Q$ of type $A_1$. The essential feature in their argument is the presence of a closed and co-closed (possibly $\Z_2$-twisted) \emph{non-vanishing} one-form $\theta_{JK}$ on the singular stratum $Q$. In the untwisted setting, the existence of $\theta_{JK}$ is equivalent to the condition that $Q$ fibers over the circle \cite{tischler} - a strong restriction on its topology. In the twisted case, it is equivalent to a $\Z_2$-cover of $Q$ fibering over $\sph^1$.

This picture is consistent with our construction. The Joyce-Karigiannis $\Z_2$-twisted one-form $\theta_{JK}$ is really a commuting $SU(2)$-Higgs field on $Q$, since $\mathfrak{h}_{SU(2)} = \mathfrak{u}(1) \cong \R$ so sections of $\mathcal{E}_W = T^*Q\otimes\mathfrak{u}(1)/\Z_2$ are exactly the $\Z_2$-twisted forms. The closed and co-closed conditions are just the last two equations in \ref{pwequations}.

Note that $\theta_{JK}$ is required to be non-vanishing. It was suggested in \cite{joycekarigiannis} that an isolated vanishing point should lead to an isolated conical singularity. In our language, ``vanishing'' means that there is a point in $Q$ where the two sheets of the spectral cover intersect. In other words, the Higgs field is ramified at that point. 

Now, to relate with the proposal in \cite{danonao}, we recall that a spectral cover $\widetilde{S}_\theta \to Q$ is determined by a Higgs field $\theta$, which is related to a positive branched section $h_\theta$ encoding the $G_2$-structure of the total space. The crucial point is equation \ref{duuuh} above, which says that \emph{locally}, the sheets of $\widetilde{S}_\theta \subset T^*Q$ are identified with the graph of $dh_\theta$. In particular, the spectral cover is Lagrangian in $T^*Q$. Now, recall that in the setup of \cite{donaldson} the ramifying locus of $h_\theta$ is a knot or link $K \subset Q$ over which the coassociative fibers acquire conical singularities. In our picture, $K$ is a subset where some sheets of $\widetilde{S}_\theta$ intersect. According to \cite{donaldson} it is expected that, at least in the so-called adiabatic limit, given two different points $x, y \in K$ one could associate to $h_\theta$ a certain ``matching gradient path'' $\gamma : [0,1] \to Q$ connecting $\gamma(0)=x$ and $\gamma(1)=y$, such that the union of vanishing cycles over $\gamma$  is an associative $\sph^3$. Deforming $K$ along such a path would bring $x$ and $y$ together, creating a point-like singularity. Geometrically, this picture describes a ``collision'' of two singular Kovalev-Lefschetz fibers producing an isolated singularity via a ``blow-down'' of the associative $\sph^3$.

The interpretation in the spectral cover picture is the following: if a matching gradient path between $x$ and $y$ exists, then \emph{$\widetilde{S}_\theta$ has non-trivial ramification \textbf{of the same order} at $x$ and $y$}. Here, the matching condition guarantees that vanishing cycles match bijectively between neighborhoods of $x$ and $y$. Notice that any subset of sheets only need to intersect at one of the points, since $h_\theta\circ \gamma$ can interpolate between components midway along $\gamma$ and thus match vanishing cycles from different components. 

Assuming for simplicity that $\theta$ is unramified on an open neighborhood $U$ of $\gamma((0,1))$, we can unambiguosly label the different sheets of $\widetilde{S}_\theta$ over $U$; call $\mathfrak{S}$ the set of such sheets. Assume also that the ramification loci over $x$ and $y$ consists of single points. For $p \in \overline{U}$, let $\textbf{Ram}_p$ denote the subset of $\mathfrak{S}$ of sheets containing a ramification point over $p$. Finally, assume that $\textbf{Ram}_x \neq \textbf{Ram}_y$. Then, if there is a deformation of $(K,h_\theta)$ bringing $x$ and $y$ together, these sheets come together in the limit point forming an enhanced singularity, described by the enhanced ramification of the deformed $\theta$ at the limit point. Clearly, this setup cannot be engineered with group $SU(2)$, since the only non-regular element of $\mathfrak{su}(2)$ is $0$.

Although most of this discussion is non-rigorous, our main point here is that this establishes a conceptual relation between the constructions in \cite{joycekarigiannis} and \cite{donaldson}, and this is our proposed answer to question $(8.\text{v})$ in the first paper. We should also mention that generically, one would expect the spectral cover to ramify at finitely many points, so isolated singularities arising from collision of links seem to be rather special from this perspective.

In \cite{galerona} we provided an explicit local construction of a ramified Higgs field arising from collision of links (see equations $(3.44)$ and $(3.45)$ in that paper). The example is in $\R^3$ with gauge group $SU(3)$. The spectral cover ramifies over the x-axis and y-axis and vanishes at the origin. In the $G_2$-picture, this is expected to describe a collision of two $A_1$-singularities.

Another interesting feature of the spectral cover interpretation is a possible Floer-theoretic description of associative spheres. Assuming that in the adiabatic limit all associatives can be obtained through matching vanishing cycles as we outlined, one could attempt to count them by looking at the differential in the Morse-Novikov cohomology generated by the critical points. From the spectral cover perspective, one looks at the Lagrangian Floer cohomology $HF^\bullet(Q,\theta(Q))$, i.e. one should count $J$-holomorphic disks bounded by sheets of the spectral cover. Since the Morse-Novikov complex computes Lagrangian Floer cohomology of $T^*Q$, the two counts must agree. We recall that Joyce's proposed definition of ``$G_2$-quantum cohomology'' \cite{qcoh} relies on the strong similarities between counting associative spheres and counting $J$-holomorphic discs bounded by Lagrangians. The upshot of our discussion is that for  coassociative ALE fibrations, the two constructions are not merely similar but \emph{exactly equal}. Hence, the quantum geometry of this class of $G_2$-manifolds can be probed with existing techniques, and as such they provide a good testing ground for Joyce's conjectures.

\section{Future Directions}

\begin{enumerate}

\item \textbf{$G_2$-metrics:} The most natural extension of this work would be to identify which fibers of the deformation family have \emph{integrable} $G_2$-structures. The proposal of Donaldson \cite{donaldson} to attack the torsion-free condition consists in solving the adiabatic limit and then perturbing the solution via a formal power series. The adiabatic solution consists of a spectral cover stationary under mean curvature flow. It would be interesting to understand what the perturbations mean geometrically.

\item \textbf{Higgs Bundles and Branched Covers:} A theorem of Hilden-Montesinos \cite{monte} says that every closed orientable 3-manifold is a 3-fold branched cover over $\sph^3$ with branched set over some knot $K$. It would be interesting to understand the behavior of harmonic Higgs fields under pullbacks/pushforwards by such covering maps. Once ramified spectral data is well-understood in this context, this could be a potential source of several interesting examples.

Related to this, consider the following: it is known that $\G_6$ is a twofold cover of $\sph^3$ branched over the Borromean rings $L$. Is there a spectral cover profile that coincides with this covering? The author believes this question is related to possible topological transitions between $G_2$-spaces \cite{aganagic}. Essentially, the idea is that if one could construct a coassociative $A_1$-fibration over the flat Borromean orbifold $(\sph^3,L)$ with $G_2$-structure $\varphi$, then a deformation of the branching locus $L$ induces a deformation of $\varphi$ itself. One could imagine deforming $L$ into a wedge of three circles such that the Higgs field becomes ill-defined at the center, corresponding to an isolated conical singularity. We then deform $L$ again, but now the three unknots are fully linked. The new spectral cover should correspond to the homology sphere $\sph^3/Q_8$, where $Q_8$ is the quaternion group. The spectral cover has transitioned from a flat to a spherical geometry. We will provide further evidence for this transition in \cite{barb}.

\item \textbf{Derived Integrable Systems?} We started this work by claiming our main construction is a $G_2$ analogue of the works of Szendr\H{o}i \cite{szendroi} and Diaconescu-Donagi-Pantev \cite{ddp} on families of ADE-fibered Calabi-Yau threefolds. However, the result of \cite{ddp} is much stronger: there is an isomorphism of \emph{integrable systems} between the Jacobian fibration of the Calabi-Yau family and the dual Hitchin fibration. The integrable system structure is encoded in the geometry of the base of the fibrations. One way to state the theorem is to say that the two bases are isomorphic special K\"ahler manifolds. Note that a crucial reason this is possible is the fact that the Hitchin moduli space is hyperk\"ahler: besides its natural K\"ahler structure, there is a second K\"ahler structure (with respect to a different complex structure) corresponding to the character variety of the base Riemann surface. 

In our setup, the same statement is hopeless, since character varieties of three-manifolds are not symplectic. However, according to \cite{ptvv}, they admit \emph{$(-1)$-shifted symplectic structures}, a generalization of symplectic structures in the context of derived geometry. Moreover, in the case of compact $G_2$-manifolds, Karigiannis and Leung \cite{leung} prove that $\moduli^\C_{G_2}$ is K\"ahler and admits a Lagrangian fibration over an \emph{affine Hessian manifold} (a real version of special K\"ahler). One is then led to ponder if there is a statement in derived geometry that would relate the K\"ahler and shifted symplectic structures, effectively completing our analogy.

\end{enumerate}

\end{document}